\documentclass[leqno,12pt]{article}

\usepackage{epsfig}

\usepackage{amsmath}
\usepackage{amscd}
\usepackage{amsopn}
\usepackage{amsthm}
\usepackage{amsfonts,amssymb}\usepackage{amsfonts,bbm}
\usepackage{latexsym}

\makeatletter
\@addtoreset{equation}{section}
\makeatother

\newtheorem{lemma}{LEMMA}[section]
\newtheorem{proposition}[lemma]{PROPOSITION}
\newtheorem{corollary}[lemma]{COROLLARY}
\newtheorem{theorem}[lemma]{THEOREM}
\newtheorem{remark}[lemma]{REMARK}
\newtheorem{remarks}[lemma]{REMARKS}

\newtheorem{examples}[lemma]{EXAMPLES}
\newtheorem{definition}[lemma]{DEFINITION}

\newtheorem*{ass}{Assumption}

\newcommand{\real}{\mathbbm{R}}

\newcommand{\nat}{\mathbbm{N}}

\newcommand{\limn}{\lim_{n \to \infty}}

\renewcommand{\a}{\alpha}
\renewcommand{\b}{\beta}

\newcommand{\g}{\gamma}

\newcommand{\vp}{\varphi}
\newcommand{\ve}{\varepsilon}

\newcommand{\reald}{{\real^d}}

\newcommand{\on}{\quad\text{ on }}
\newcommand{\und}{\quad\mbox{ and }\quad}
\newcommand{\inv}{^{-1}}
\newcommand{\ov}{\overline}

\newcommand{\V}{\mathcal V}  
\newcommand{\W}{\mathcal W}  
\newcommand{\C}{\mathcal C}  

\newcommand{\F}{\mathcal F}

\newcommand{\B}{\mathcal B}
\renewcommand{\S}{\mathcal S}

\newcommand{\K}{\mathcal K}
\newcommand{\Q}{\mathcal Q}

\newcommand{\dist}{\mbox{\rm dist}}

\newcommand{\convex}{\operatorname*{conv}}

\newcommand{\itemframe}%
{\setlength{\parskip}{10pt}\begin{enumerate} \setlength{\topsep}{10pt}%
\setlength{\itemsep}{15pt}\setlength{\parsep}{5pt}}

\newcommand{\vx}{\ve_x}

\newcommand{\Px}{\mathcal P(X)}

\newcommand{\uc}{{U^c}}

\newcommand{\vc}{{V^c}}
\newcommand{\wc}{{W^c}}

\newcommand{\lam}{\lambda}

\newcommand{\kap}{\operatorname*{cap}}

\newcommand{\mc}{m_{\mbox{\small\,cap}}}

\newcommand{\ppt}{\mathbbm P=(P_t)_{t>0}}
\newcommand{\vvl}{\mathbbm V=(V_\lambda)_{\lambda>0}}
\newcommand{\liml}{\lim_{\lambda\to \infty}} 
\newcommand{\ev}{E_{\mathbbm V}}  
\newcommand{\es}{E_{\mathbbm P}}  
\newcommand{\lvl}{\lambda V_\lambda}

\newcommand{\U}{{}^U\!}

\newcommand{\du}{\delta_U}
\newcommand{\gr}{G_0 }

\title{Unavoidable sets and\\  harmonic measures  living on small sets} 
\author{WOLFHARD HANSEN and IVAN NETUKA
\thanks{Both authors gratefully acknowledge support
by CRC-701, Bielefeld.}}
\date{}
\begin{document}
\maketitle

\begin{abstract}
Given a  connected open  set $U\ne\emptyset$ in $\mathbbm R^d$, $d\ge 2$, 
a relatively closed set~$A$ in~$U$ is called \emph{unavoidable in~$U$}, if
Brownian motion,  starting in $x\in U\setminus A$ and killed when leaving~$U$,  hits~$A$ 
almost surely or, equivalently,  if  the  harmonic measure for~$x$ with respect to~$U\setminus A$ has mass $1$ on $A$.
First a new criterion for unavoidable sets is proven  
 which  facilitates the construction of smaller and smaller unavoidable sets in $U$.  
Starting with an arbitrary champagne subdomain of~$U$
(which is obtained omitting  a~locally finite union of pairwise disjoint closed balls 
$\overline B(z, r_z)$, $z\in Z$,  
 satisfying $\sup_{z\in Z} r_z/\dist(z,U^c)<1$),  
a~combination of the criterion and the existence of small nonpolar compact sets of Cantor type yields a~set~$A$ 
on which harmonic measures for $U\setminus A$ are living and which has Hausdorff dimension~$d-2$ and, if~$d=2$,  
logarithmic Hausdorff dimension~$1$.

This can be done as well for Riesz potentials (isotropic $\alpha$-stable processes) on Euclidean space and for censored stable processes
on $C^{1,1}$ open subsets. Finally,  in the very general setting
of a balayage space $(X,\mathcal W)$ on which the function $1$ is harmonic (which  covers not only large classes of   
second order partial differential equations, but also non-local situations as, for example, given by  Riesz potentials,
isotropic unimodal L\'evy processes 
or censored stable processes) a~construction of champagne subsets~$X\setminus A$ of~$X$ with
 \emph{small} unavoidable sets~$A$ is given which generalizes (and partially improves) 
recent constructions in the classical case.

 {

  }
\end{abstract}

\section{Introduction}

This paper is devoted to the construction of \emph{small} unavoidable sets in various potential-theoretic settings
(classical potential theory, Riesz potentials (isotropic $\a$-stable processes), censored $\a$-stable processes,
 harmonic spaces, balayage spaces). In particular, we shall give optimal answers to the question of how small a~set      
 may be on~which harmonic measures is living.

 For the moment, let $U$ be a non-empty connected  open  set in $\reald$, $d\ge 2$. If $d=2$ we assume
that the complement of $U$ is nonpolar (otherwise our considerations become trivial).
A~relatively closed subset  $A$ of $U$  is  called \emph{unavoidable in~$U$}
if~Brownian motion, starting in $U\setminus A$ and killed when leaving $U$,  hits $A$ almost surely or, equivalently,
if~$\mu_y^{U\setminus A}(A)=1$,  for every $y\in U\setminus A$, where $\mu_y^{U\setminus A}$ 
denotes the harmonic measure at~$y$  with respect to~$U\setminus A$.

A \emph{champagne subdomain} of  $U$  is obtained
by omitting a union $A$ of  pairwise disjoint closed balls $\ov B(x, r_z)$, $z\in Z$, 
the \emph{bubbles}, 
where $Z$ is an infinite, locally finite set in $U$, $\sup_{z\in Z} r_z/\dist(z,\partial U)<1$ and,
if $U$ is unbounded, the radii $r_z$ tend to $0$ as~$z\to \infty$.    It will sometimes be convenient to write $Z_A$ instead of~$Z$.

For $r>0$, let 
$$
                     \kap(r):=\begin{cases}
                                r^{d-2} ,&\quad\mbox{ if }d\ge 3,\\[1mm]
                      (\log^+ \frac 1r)\inv  ,&\quad\mbox{ if } d=2.            
                   \end{cases}                   
  $$
Recently, the following has been shown (see \cite[Theorem 1.1]{HN-champagne};
cf.\  \cite{gardiner-ghergu,pres} for the case $U=B(0,1)$ and $h(t)=(\kap(t))^\ve$).

\begin{theorem}\label{main-old}
Let $h\colon (0,1)\to (0,\infty)$ be such that  $\liminf_{t\to 0} h(t)=0$.     
Then, for every $\delta>0$, there is a champagne subdomain  $U\setminus A$ of  $U$ such that $A$~is unavoidable  in~$U$
and $\sum_{z\in Z_A} \kap(r_z)h(r_z)<\delta$.
\end{theorem}

We note that, for every champagne subdomain $U\setminus A$ of $U$ with unavoidable $A$ the series
$\sum_{z\in Z_A} \kap(r_z)$ diverges. This follows by a slight modification of arguments in \cite{gardiner-ghergu,pres}
(for the possibility of omitting finitely many bubbles, see (1b) in Lemma~\ref{simple}). Therefore the  condition
$\liminf_{t\to 0} h(t)=0$ is also necessary for the conclusion in Theorem \ref{main-old}.

The proof of Theorem \ref{main-old} given in \cite{HN-champagne} is based on a criterion
for unavoidable sets which, in probabilistic terms,  relies on the continuity
of the paths for Brownian motion (see \cite[Proposition 2.1]{HN-champagne}).
We shall use a criterion which, using entry times  
$T_E(\omega):=\inf \{t\ge 0\colon X_t(\omega)\in E\}$ for Borel measurable sets~$E$,  states the following.

\begin{proposition}\label{unavoidable}
Let $A$ and $B$ be relatively closed subsets of $U$ and $\kappa>0$ such that $A$ is unavoidable
in $U$ and $P^x[T_B<T_\uc]\ge \kappa$, for every $x\in A$. Then $B$~is unavoidable in $U$.
\end{proposition} 

Such a criterion holds as well for very general transient Hunt processes
 on locally compact spaces $X$ with countable base (cf.\ Proposition \ref{AB-unavoidable} and its proof).
Iterated application may lead to very small unavoidable sets. 

Starting, for example, in our classical
case with an arbitrary champagne subdomain of $U$     with unavoidable union 
of bubbles (obtained by Theorem \ref{main-old} or more simply by directly using Proposition \ref{unavoidable}
twice), an application of the new criterion quickly leads to the following result
on the smallness of sets, where harmonic measures may live (cf.\  Corollary  \ref{loga2}).

\begin{theorem}\label{loga}
There exists a relatively closed set $A$ in $U$ having the following properties:
\begin{itemize} 
\item The open set $U\setminus A$ is connected.
\item  For every  $x\in U\setminus A$, $\mu_x^{U\setminus A}(A)=1$.
\item The set $A$ has Hausdorff dimension $d-2$ and, if $d=2$,  logarithmic Hausdorff dimension $1$.
\end{itemize} 
\end{theorem}

Let us note that smaller Hausdorff dimensions are not possible, since any set having strictly positive 
harmonic measure has at least Hausdorff dimension~$d-2$ and, if $d=2$, logarithmic
Hausdorff dimension~$1$ (see (\ref{infty}) and the subsequent lines).

A general equivalence  involving arbitrary measure functions $\phi$ is presented in Theorem \ref{main}.
Moreover, there are analogous results for Riesz potentials (isotropic $\a$-stable processes) on Euclidean space
(see Section \ref{section-riesz}).

On the other hand,   P.\,W.\  Jones and T.\,H.\ Wolff~\cite{jones-88}  proved 
that harmonic measures for planar domains are always living on sets of Hausdorff  dimension 
 at most~$1$. Later T.\,H.\ Wolff \cite{wolff-93} refined this 
by showing that there always exists a~set which has full harmonic measure and
$\sigma$-finite  $1$-dimensional measure.
 For simply connected domains, N.\,G.\ Makarov \cite{makarov-85-hausdorff} 
showed that any set of Hausdorff dimension  strictly less than $1$ has zero harmonic 
measure. In fact, he found an optimal measure function such that harmonic measures 
are  absolutely continuous with respect to the corresponding Hausdorff measure. 
Further results for planar domains may be found in \cite{carleson-73,oksendal-81,
oksendal-80, carleson-82, kaufman-82, makarov-85-hausdorff,carleson-85,pommerenke-86,
makarov-87-conformal,makarov-87-metric, 
bishop-90,volberg-93,jones-makarov-95,batakis-96,makarov-98}.
For higher dimensions $d\ge 3$,  J.\ Bourgain \cite{bourgain-87} proved that there exists an absolute 
constant $\gamma(d)  < d$ such that harmonic measures for open sets in $\reald$  always have full mass
on a~set of dimension at most $\gamma(d)$. As shown by T.\,H.\ Wolff \cite{wolff-95},
$\gamma(3)>2$.

In Section \ref{champagne-bal} we shall prove that champagne subsets with small unavoidable unions
of bubbles exist in very general settings, where we have a strictly positive Green function $G$
and a capacity function which is related to the behavior of $G$ close to the diagonal. This even
simplifies the construction given for Theorem \ref{main-old} (in the case, where $\lim_{t\to 0} h(t)=0$)
and has applications to large classes of elliptic second order PDE's as well as to Riesz potentials, 
isotropic unimodal L\'evy processes and censored stable processes (see Examples \ref{examples}).

For the convenience of the reader, we add an Appendix. In a first part we discuss  balayage spaces     
and explain their relationship with Hunt processes, sub-Markov semigroups, sub-Markov resolvents.   
In a second part we give a self-contained proof
for the construction of small nonpolar compact sets of Cantor type (Theorem \ref{existence}).

\thanks{We are indebted to Moritz Ka\ss mann for stimulating discussions  
and for having raised the question of an application to censored stable processes.}

\section{Unavoidable sets}\label{unav}

To work in reasonable generality let us consider a balayage space $(X,\W)$
such that the function $1$ is harmonic, and let 
$\mathfrak X=(\Omega,\mathfrak M, \mathfrak M_t, X_t,\theta_t,P^x)$ 
be an associated Hunt process  (see \cite{BH} or the Appendix).
This covers large classes of second order partial differential equations as well as   
non-local situations as, for example, given by Riesz potentials or censored stable
processes.

We shall proceed in such a way  that the reader who is mainly interested 
in classical potential theory may look at most of the following assuming that 
$X$ is a~connected open set $U$ in $\reald$, $d\ge 2$ ($\real^d\setminus U$
nonpolar if $d=2$),  $\W=\S^+(U)\cup\{\infty\}$, where $\S^+(U)$  is the set of all 
positive superharmonic functions on $U$, and $\mathfrak X$ is Brownian motion on $U$.

Let $\C(X)$ be the set of all continuous real functions on $X$ and let
$\K(X)$ denote  the set of all functions in  $\C(X)$ having compact support.
Moreover,  
let $\B(X)$ be the set of all Borel measurable numerical functions on $X$.
Given any set $\F$ of functions, let~$\F^+$ ($\F_b$, respectively)          
denote the set of all positive (bounded, respectively) $f\in \F$. 

We first recall some basic  notions and facts on balayage we shall need. For
numerical functions~$f$ on~$X$, let
$$
         R_f:=\inf \{u\in\W\colon u\ge f\}.
$$
In particular, for every  $A\subset X$ and $u\in\W$, let
$$
       R_u^A:=R_{1_Au}=
\inf\{v\in \W\colon v\ge u\mbox{ on } A\}.
$$
Let $\Px$ denote the set of all potentials in $\C(X)$, that is,
\begin{equation}\label{potential}
\Px:=\{p\in \W\cap C(X)\colon \inf
\{R_p^\vc\colon V\mbox{ relatively compact}\}=0\}.
\end{equation} 
Then $\W$ is the set of all limits of increasing sequences in $\Px$.   
A potential $p\in \Px$ is called \emph{strict} if any two measures $\mu,\nu$ on~$X$ satisfying $\int p\,d\mu=\int p\,d\nu<\infty$
and $\int q\,d\mu\le \int q\,d\nu$ for all $q\in\Px$, coincide.

For every $x\in X$ and $A\subset X$, 
there exists a unique measure $\vx^A$ on $X$
such that 
$$
          \int u \,d\vx^A= R_u^A(x), \qquad  \mbox{ for every }u\in \W.
$$
Of course, $\vx^A=\vx$,
if $x\in A$. We note that in \cite{BH} the measure $\vx^A$
is denoted by~${\overset \circ \ve}{}_x^A$, whereas there $\vx^A$ denotes the
\emph{swept} measure defined by $\int u \,d\vx^A=\hat R_u^A(x)
:=\liminf_{y\to x} R_u^A(y)$, $u\in \W$ (it coincides with
$\overset\circ\ve{}_x^A$, if $x\notin A$).

In terms of the associated
Hunt process, the measures $\vx^A$, $x\in X$ and 
$A$ Borel measurable,  are the distributions of the process
starting in $x$ at the time $T_A$ of the first entry into $A$
(which is defined  by
$T_A(\omega):=\inf\{t\ge 0\colon X_t(\omega)\in A\}$).
 
If we have to emphasize  the ``universe'' $X$ to avoid 
ambiguities, we shall add a~superscript $X$ and, for example,
write ${}^X\!R_1^A$ and  ${}^X\!\vx^A$ instead of $R_1^A$ and  $\vx^A$.

Given a  finite measure $\nu$ on $X$, let $\|\nu\|$ denote its total mass.    

\begin{definition}
Let $A$ be  a subset of $X$. It  is called \emph{unavoidable} {\rm(}in $X${\rm)}, if~$\|\vx^A\|=1$, for every $x\in X$, 
 or, equivalently, if          
\begin{equation}\label{def-unavoidable}
             R_1^A=1.       
\end{equation} 
In probabilistic terms {\rm(}if $A$ is Borel measurable{\rm)}:  The set $A$ is unavoidable, if 
\begin{equation}\label{def-prob}
  P^x[T_A<\infty] =1,\qquad\mbox{ for every }x\in X.
\end{equation} 

The set $A$ is called \emph{avoidable}, if $R_1^A$ is not identically $1$, that is {\rm(}provided $A$ is Borel measurable{\rm)},
 if there exists a point $x\in X$ such that  $ P^x[T_A<\infty] <1$. 
\end{definition} 

This is consistent with the definition we used           
in the Introduction, where we consider classical potential theory on $\reald$,  
a connected open set $U\ne\emptyset$, and a relatively closed subset $A$ of $U$.
Indeed, from a probabilistic point of view,
the consistency is obvious, since  $\|{}^U\!\vx^A\|$
is the probability that Brownian motion killed upon leaving $U$
enters $A$ during its lifetime.                                                           
For an analytic proof, we recall that, for  $x\in U\setminus A$,
\begin{eqnarray*} 
     {}^U\!R_1^A(x)
        &\!\!=\!\!&\inf\{ u(x)\colon u\in\S^+(U),
           \ u\ge 1 \mbox{ on }A\},\\[2mm]
     \mu_x^{U\setminus A}(A)
        &\!\!=\!\!& \inf\{v(x)\colon v\in \S^+(U\setminus A),\ 
     \liminf\nolimits_{y\to z}v(y)\ge 1
          \mbox{ for }z\in A\cap \partial (U\setminus A)\}
\end{eqnarray*} 
(cf.\ \cite[Chapter 6]{gardiner-armitage}). Hence, trivially,
 ${}^U\!R_1^A(x)\ge  \mu_x^{U\setminus A}(A)$. If 
$v\in \S^+(U\setminus A)$ such that $ \liminf\nolimits_{y\to z}v(y)\ge
1$ for every $z\in A\cap \partial (U\setminus A)$, then the
function, which is equal to~$1$ on~$A$ and equal to $\min\{1,v\}$
on~$U\setminus A$, 
is superharmonic on $U$. This yields the reverse inequality.       
Hence $R_1^A(x)=1$ if and only if $\mu_x^{U\setminus A}(A)=1$. 

Returning to the setting of the balayage space $(X,\W)$, where the function $1$ is harmonic,   
we observe some  elementary facts.

\begin{lemma}\label{simple}
\begin{enumerate}
\item
For every unavoidable  set  $A$ the following holds:
\begin{itemize}
\item[\rm(a)] Every   set $A'$ in $X$ containing $A$
   is unavoidable.
\item[\rm(b)] For every  relatively compact  set $F$ in $X$, the
  set $A\setminus F$ is unavoidable.
\item[\rm(c)] If $A$ is the union of relatively compact sets $F_n$, $n\in\nat$,
then, for all $x\in X$, the series $\sum_{n\in\nat} \|\vx^{F_n}\|$ diverges.
\end{itemize} 
\item For every relatively compact open set $V$ in $X$, the set
$X\setminus V$ is unavoidable. 
\end{enumerate}
\end{lemma} 

\begin{proof} 1. (a) Trivial consequence of $R_1^A\le R_1^{A'}$.

(b) 
Let $\vp\in\K(X) $ 
such that $1_F\le \vp$. Then $p:=R_\vp\in \Px$ (see \cite[II.V.2]{BH}). 
Let $u\in \W$ such that $u\ge 1$ on $A\setminus F$. Then 
$u+p\in\W$ and $u+p\ge 1$ on $A$. So $u+p \ge R_1^A =1$,        
that is, $u-1\ge -p$. Since the function $1$ is harmonic, the function $u-1$ is hyperharmonic
and hence, by the minimum principle, $u-1\ge 0$ (see~\cite[III.6.6]{BH}). So $u\ge 1$ proving that $R_1^{A\setminus F}=1$.

(c) Let  $m\in\nat$ and let $B$ denote the union of all $F_n$, $n\ge m$.
By (b), the set $B$ is unavoidable. Hence $1=R_1^B\le \sum\nolimits_{n\ge m} R_1^{F_n}$.

2. A consequence of (b) taking $A=X$, $F=V$.
\end{proof}




Iterated application of Proposition \ref{AB-unavoidable},2 will help us to construct  small
 unavoidable sets.

\begin{proposition}\label{AB-unavoidable}
For all subsets $A,B$ of $X$ the following holds:\begin{itemize} 
\item[\rm 1.] $A$ is unavoidable or $\inf_{x\in X} R_1^A(x)=0$.
\item[\rm 2.]  
  If $A$ is unavoidable, $\kappa>0$,  and $ R_1^B\ge \kappa$ on $A$, then $B$ is unavoidable.
\item[\rm 3.] Suppose that  $(X,\W)$ has the \emph{Liouville property}, that is,
every bounded harmonic function on $X$ is constant. Then
\begin{itemize} 
\item[\rm (a)]  $A$ is avoidable if and only if $\hat R_1^A$ is a potential,
 \item[\rm (b)]   $A\cup B$ is avoidable if and only if $A$ and $B$ are avoidable.
\end{itemize} 
\end{itemize} 
\end{proposition}

\begin{proof} 1.  Of course,  $\g:=\inf_{x\in X} R_1^A(x)\in [0,1]$. 
If $u\in \W$ such that $u\ge 1$ on $A$, then $u\ge R_1^A\ge \g$, hence $u-\g\in \W$ and $u-\g\ge 1-\g$ on $A$.
So $R_1^A -\g\ge  R_{1-\g}^A$. 
If~$u\in \W$ such that $u\ge 1-\g$ on $A$, then $u+\g\in \W$ and $u+\g\ge 1$ on~$A$.
Therefore $R_{1-\g}^A + \g\ge R_1^A$. Thus
\begin{equation}\label{gg1}
R_1^A=  R_{1-\g}^A +\g.
\end{equation} 
Since $R_{1-\g}^A= (1-\g) R_1^A$, (\ref{gg1}) shows that $ \g R_1^A=\g$.  So $\g=0$ or  $R_1^A= 1$. 

2.  Suppose that $R_1^B\ge \kappa>0$ on $A$, and let $u\in \W$ such that $u\ge 1$ on $B$.
Then $u\ge R_1^B\ge \kappa$ on $A$, and therefore $u\ge R_\kappa^A$. Hence $R_1^B\ge R_\kappa^A=\kappa R_1^A$.
If  $R_1^A=1$, we obtain that $R_1^B\ge \kappa$, and hence $B$ is unavoidable by (1).

3. (a)  Let $h$ be the greatest harmonic minorant of $\hat R_1^A$. Then $0\le h\le 1$ and $p:=\hat R_1^A-h$ is a potential.
By the Liouville property, $h$ is constant. If $A$ is avoidable, we hence see, by (1), that $h=0$ showing that
$\hat R_1^A$ is the potential $p$. 

(b)  If $A$ and $B$ are avoidable, then the  inequality  $\hat R_1^{A\cup B}\le \hat R_1^A +\hat R_1^B$ shows that $\hat R_1^{A\cup B}$
is a potential, and hence $A\cup B$ is avoidable, by (a). The converse is trivial, by (1a) in  Lemma \ref{simple}.
\end{proof}

For an application in classical potential theory (and for more
general harmonic spaces), the following simple consequence will be useful
in combination with further applications of Proposition \ref{AB-unavoidable},2.

\begin{proposition}\label{harmonic} 
Suppose that $V_n$, $n\in\nat$,
are relatively compact open  sets covering $X$ such that,
for every $x\in V_n$, the harmonic measure $\mu_x^{V_n}=
\vx^{X\setminus V_n}$ is supported by the boundary $\partial V_n$.
Then the union of all boundaries $\partial V_n$, $n\in\nat$,
is unavoidable. 
\end{proposition} 

\begin{proof} Given $x\in X$, there exists $n\in\nat$ such that $x\in
  V_n$, and then, by  \cite[VI.2.4 and VI.9.4]{BH},    
 $$
\vx^{\partial V_n}=(\vx^{X\setminus V_n})^{\partial    V_n}
  =\vx^{X\setminus V_n}.
$$ 
Since the measures $\vx^{X\setminus V_n}$ have total mass $1$, 
by Lemma \ref{simple}, an application of Proposition \ref{AB-unavoidable},2  (with $A=X$ and $\kappa=1$) finishes the proof.
\end{proof}

For Riesz potentials (isotropic $\a$-stable processes) on Euclidean space    
we obtain the following (the reader who is primarily interested in classical potential theory     
may pass directly to Section \ref{living}).

\begin{proposition}\label{riesz}
Let $X=\reald$, $d\ge 1$, $0<\a<\min\{d,2\}$, and let $\W$ be the
set of all increasing limits of Riesz potentials $G\mu\colon
x\mapsto \int |x-y|^{\a-d}\,d\mu(y)$ {\rm(}$\mu$~finite measure on
$\reald$ with compact support\,{\rm)}. Moreover, let 
$0<R_1<R_2<\dots$ be such that $\limn R_n = \infty$.
Then the following holds:
\begin{itemize} 
\item[\rm (1)] If $d\ge 2$ and $\a>1$, then the union of all
$\partial B(0,R_n)$, $n\in\nat$, is unavoidable.
\item[\rm (2)] For every $\delta>0$, the union of all shells
$B_n:=\{x\in \reald\colon R_n\le |x|\le (1+\delta)R_n\}$, $n\in\nat$,
is unavoidable.
\end{itemize} 
\end{proposition}  

\begin{proof} (1) The boundary $S:=\partial B(0,1)$, $n\in\nat$, is not ($\a$-)thin 
at any of its points (cf.~\cite[VI.5.4.4]{BH}),  and hence $R_1^S\in \W$ 
(in fact, $R_1^S\in \Px$). In particular,
\begin{equation*} 
     \kappa:=\inf\{R_1^S(x)\colon x\in B(0,1)\}>0.
\end{equation*} 
By scaling invariance, $\inf\{ R_1^{\partial B(0,R_n)}(x)\colon x\in B(0,R_n)\}=\kappa$, for every $n\in\nat$.
Since every $x\in\reald$ is contained in some $B(0,R_n)$ (and $\reald$ is unavoidable), 
we see, by~Proposition \ref{AB-unavoidable},2, that the union of all $\partial B(0,R_n)$,
$n\in\nat$, is unavoidable.

(2) By scaling invariance, for every $n\in\nat$,  
\begin{equation*} 
     \inf\{R_1^{B_n}(x)\colon x\in B(0,R_n)\}=     \inf\{R_1^{\{1\le |\cdot|\le 1+\delta\}} (x)\colon x\in B(0,1)\}>0.
\end{equation*} 
As in the proof of (1) we now obtain that the union of all $B_n$, $n\in\nat$, is unavoidable.
\end{proof}

\begin{remark}{\rm
If $\a\le 1$ ($\a<1$ for $d=1$), then, for every $R>0$,
the boundary $\partial B(0,R)$ is ($\a$-)polar and hence
$\vx^{\partial B(0,R)}=0$, for every $x\in \reald\setminus
\partial B(0,R)$ (cf.~\hbox{\cite[VI.5.4.4]{BH}}). So statement
of (1) in Proposition \ref{riesz} does not hold.
}
\end{remark} 

For a general balayage space,  we can still say the following.

\begin{proposition}\label{shells}
Let $(U_m)$ be an exhaustion of $X$,  that is,
let $(U_m)$ be a~sequence of relatively compact open sets in $X$ such that
$\ov U_m\subset U_{m+1}$, $m\in\nat$, and $X=\bigcup_{m\in\nat} U_m$.
Let $(k_n)$ be a sequence of natural numbers.

Then there exist $m_n\in\nat$  such that $m_{n}+k_n\le m_{n+1}$,
$n\in\nat$,
and the union~$B$ of the compact   ''shells'' 
$B_n:=\ov U_{m_{n+1}}\setminus  U_{m_n+k_n}$, $n\in\nat$,
is unavoidable.
\end{proposition} 

\begin{proof} We start with $m_1:=1$.
 Let $n\in\nat$ and suppose that $m_n$ has been
chosen. We define 
$$
           A_n:=\ov U_{m_n} \und  V_n:=U_{m_n+k_n}.
$$
The functions 
$h_m\colon x\mapsto \vx^{X\setminus V_n}(\ov U_m\setminus V_n)$,
 $m\ge m_n+k_n$, which are continuous  on $V_n$ (see \cite[VI.2.10]{BH}), 
are increasing to $1$. 
So there exists $m_{n+1} \ge m_n+k_n$ such that $h_{m_{n+1}}>1/2$ on the
compact $A_n$ in $V_n$, and hence  
$$
  \| \vx^{B_n}\|\ge \ve_x^{X\setminus V_n}(B_n)=h_{m_{n+1}}(x)>1/2, 
\qquad\mbox{ for every }x\in A_n.
$$
The claim of the proposition follows from 
Proposition \ref{AB-unavoidable},2, since
the union of the sets $A_n$, $n\in\nat$, is the whole space $X$, 
which, of course, is unavoidable.
\end{proof}

\section{Harmonic measures living on small sets}\label{living}

As in our Introduction, let $U\ne \emptyset$ be
a~connected open set in $\reald$, $d\ge 2$, such that 
$\real^d\setminus U$ is nonpolar if~$d=2$,
and let us consider classical potential theory on~$U$
(Brownian motion killed upon leaving $U$). 

\begin{proposition}\label{F-replace}
Let $U\setminus A$ be a champagne subdomain of $U$ such that 
$A$ is unavoidable in $U$. 
Then the following holds.
\begin{enumerate}
 \item[\rm (1)]
 For every   nonpolar compact  $F$  in $B(0,1)$,   the union $B$ of all 
sets $z+r_z F$, $z\in Z_A$, is relatively closed and unavoidable in $U$.
 \item[\rm (2)]
If $\b\in (0,1)$ and $B$ denotes the union of all $\ov B(z,\b r_z)$,
$z\in Z_A$, 
then $U\setminus B$ is a~champagne subdomain of $U$ such that 
$B$ is unavoidable.
\end{enumerate}
\end{proposition} 

\begin{proof} Of course, (2) is a trivial consequence of (1).
So let $F$ be  a nonpolar compact in $B(0,1)$
and let $B$ be the union  of all compact sets $F_z:=z+r_zF$, 
$z\in Z_A$. 
Obviously,  $B$~is relatively closed in $U$. Since $\sup_{z\in Z_A} r_z/\dist(z,U^c)<1$,
there exists $\ve>0$ such that,  for every $z\in Z_A$,   
the closure of $B_z:= B(z,(1+\ve)r_z)$ is contained in $U$.
By \cite[Lemma~5.3.3 and Theorem~3.1.5]{gardiner-armitage},
$$
        \kappa:=\inf\{{}^{B(0,1+\ve)}\! \hat R_1^F(x)
               \colon x\in \ov B(0,1)\}>0.
$$
Let  $u$ be a positive  superharmonic function on~$U$ such that 
$u\ge 1$ on $B$. Then, for every $z\in Z_A$, $u\ge 1$ on 
$F_z:=z+r_zF$ and hence,  by scaling and translation invariance,  
$u\ge \kappa$ on $\ov B(z,r_z)$. Thus $u\ge \kappa$ on $A$. An application
of Proposition~\ref{AB-unavoidable},2 finishes the proof of (1).
\end{proof} 

For the existence of champagne subdomains of $U$ with unavoidable bubbles
which is  needed for an application of Proposition \ref{F-replace}, we could
use Theorem \ref{main-old}. However, let us note that using Proposition \ref{harmonic}, 
it is very  easy to  construct champagne subdomains $U\setminus B$, where
$B$ is unavoidable. Indeed, there exists $\kappa>0$ such that 
$$
            {}^{B(0,2)}\!R_1^{B(0,1/4)}\ge \kappa
               \on \ov B(0,1).
$$
Let $(V_n)$ be an exhaustion
of $U$ and     
$$
         \ve_n:=\frac 12\min\{\dist(\partial V_n, \partial V_{n-1}
                    \cup\partial V_{n+1}), 1/n\},\qquad n\in\nat
$$
(take $V_0:=\emptyset$). For every $n\in\nat$, we may choose a finite
set $Z_n$ in $\partial V_n$ such that the balls $B(z,\ve_n)$, $z\in
Z_n$, cover $\partial V_n$ and the balls $\ov B(z,\ve_n/4)$
are pairwise disjoint.
Let 
$$
A:=\bigcup\nolimits_{z\in Z_n, n\in\nat} \ov B(z,\ve_n) \und
B:=\bigcup\nolimits_{z\in Z_n, n\in\nat} \ov B(z,\ve_n/4).
$$
Then $A$ and $B$ are relatively closed in $U$ and
$U\setminus B$ is a champagne subdomain of $U$.
By Proposition \ref{harmonic} and Lemma \ref{simple}, 
$A$ is unavoidable.
Arguing similarly as in the proof of
Proposition \ref{F-replace}, we obtain that $u\ge \kappa$ on 
$A$, 
for  every positive superharmonic function $u$ on $U$ 
such that $u\ge 1$ on $B$. Hence, by Proposition \ref{AB-unavoidable},2,
$B$ is unavoidable in $U$.

Part (1) of Proposition \ref{F-replace} indicates that harmonic
measures may live on very small sets. To decide how small such
sets may really be  let us recall a few basic facts 
on measure functions  and Hausdorff measures.

Any  function $\phi\colon (0,\infty)\to (0,\infty]$ 
which is increasing and  satisfies 
 $\lim_{t\to 0}\phi(t)=0$ is called \emph{measure function}.
Given such a~function $\phi$ and a~subset~$E$ of~$\reald$, 
we define (cf.\  \cite[Definition 5.9.1]{gardiner-armitage})
$$
     M_\phi^{(\rho)}(E):=\inf\left\{\sum\nolimits_{n\in\nat} \phi(r_n)\colon E\subset 
                                 \bigcup\nolimits_{n\in\nat} B(x_n,r_n)\mbox{ and }r_n<\rho\mbox{ for each }n\right\},
$$
for $\rho\in (0,\infty)$, and 
$$
         m_\phi(E):=\lim\nolimits_{\rho\to 0} M_\phi^{(\rho)}(E).
$$
If $\phi,\psi$ are measure functions, then of course
\begin{equation}\label{pp}
                                                 m_\phi(E) \le m_\psi(E), \qquad \mbox { whenever  } \phi\le \psi\mbox{ on some interval } (0,\ve).
\end{equation} 

The \emph{Hausdorff dimension} of a  bounded set $E$ in $\reald$ is
the
 infimum of all $\g>0$ such that $m_{t^\g}(E)<\infty$
(it is at most $d$). Its \emph{logarithmic  Hausdorff dimension} is 
the infimum of all $\g>0$ such that $m_{h^\g}(E)<\infty$,
where $h(t):=(\log^+ \frac1t)\inv$ and, as usual, $\inf \emptyset :=\infty$.
If the logarithmic Hausdorff dimension of ~$E$  is finite, 
then $E$~has Hausdorff dimension~$0$.

Obviously,  $\kap$ is a measure function. If $V$ is an open set in $\reald$,
$x\in V$, and $E$~is a~Borel measurable set in $V^c$ such that
$\mu_x^V(E)>0$,
 then $E$ is nonpolar and hence
(cf.\ \cite[Theorem 6.5.5    and Theorem 5.9.4]{gardiner-armitage})
\begin{equation}\label{infty}
                                                       \mc(E)=\infty.
\end{equation} 
In particular, the Hausdorff dimension of $E$ is at least $d-2$ and, if $d=2$,  its logarithmic
Hausdorff dimension is at least $1$.

\begin{theorem}\label{main}
Let  $\phi$ be a~measure function.
Then the following statements are equivalent:   
\begin{itemize}
\item[\rm(i)]
 $\liminf_{t\to 0} \phi(t)/\kap(t)=0$.
\item[\rm(ii)]
There exists a relatively closed set $A$ in $U$ such that $m_\phi(A)=0$,
the open set $U\setminus A$ is connected, and  $\mu_x^{U\setminus A} (A)=1$, for every $x\in U\setminus A$.
\end{itemize} 
\end{theorem} 

\begin{proof} 1. Let us suppose first that (ii) holds.
Assuming that, for some $\ve,\delta>0$, $\phi/\kap \ge\delta$ on~$(0,\ve)$, we would
obtain, by (\ref{infty}),  that
$
      0= m_\phi(A)\ge \delta\, \mc (A)= \infty
$,
a contradiction.  Thus $\liminf_{t\to 0} \phi(t)/\kap(t)=0$, that is, (i) holds.

2. To prove that (i) implies (ii) we let $  h:= \phi/\kap$ and suppose 
\hbox{$\liminf_{t\to 0} h(t) =0$}. 
By Theorem \ref{existence}, there exists a nonpolar compact
 $F\subset B(0,1)$ of Cantor type such that 
$m_\phi(F)=0$ and $B(0,1)\setminus F$ is connected.

Let $B$ be any union of pairwise disjoint closed balls $B(z,r_z)$, $z\in Z$, in $U$
such that $U\setminus B$ is a champagne subdomain of~$U$ and $B$~is unavoidable.
Then the union~$A$ of all compact sets $z+r_zF$, $z\in Z$, is relatively closed in~$U$,
$U\setminus A$ is connected, and $A$ is unavoidable in~$U$, 
by Proposition \ref{F-replace} (roles of $A$ and $B$ interchanged).
\end{proof} 

Applying the implication (i)\,$\Rightarrow$\,(ii) to the measure function 
$$
\phi(t):=  t^{d-2} (\log^+ \frac 1t)\inv (\log^+\log^+ \frac 1t)^{-2}
$$
(with $\log^+ 0:=0$) we obtain a set $A$ such that 
$m_{\,\mbox{\small cap}^{1+\ve}}(A)=0$, $\ve>0$, since $\kap^{1+\ve}(t)\le \phi(t)$
for small $t$. Thus, by   (\ref{infty}), Theorem \ref{main} 
has the following immediate consequence.

\begin{corollary}\label{loga2}
There exists a relatively closed set $A$ in $U$ having the following properties:
\begin{itemize} 
\item The open set $U\setminus A$ is connected.
\item  For every  $x\in U\setminus A$, $\mu_x^{U\setminus A}(A)=1$.
\item The set $A$ has Hausdorff dimension $d-2$ and, if $d=2$,  logarithmic Hausdorff dimension $1$.
\end{itemize} 
\end{corollary}

\section{Application to Riesz potentials}\label{section-riesz}

Let us next consider Riesz potentials (isotropic $\a$-stable processes) on Euclidean space,
that~is, let $X=\reald$, $d\ge 1$, $0<\a<\min\{d,2\}$, and let $\W$ be the
set of all increasing limits of Riesz potentials $G\mu\colon
x\mapsto \int |x-y|^{\a-d}\,d\mu(y)$ {\rm(}$\mu$~finite measure on
$\reald$ with compact support\,{\rm)}. The following result
extends Theorem \ref{main-old} and Theorem  \ref{loga} (in the case $U=\reald$)
to Riesz potentials.

\begin{theorem}\label{riesz-hausdorff}
\begin{itemize}
\item[\rm (1)] Let $\phi$ be a measure function. Then the following statements
are equivalent:
\begin{itemize}
\item [\rm (i)] $\liminf_{t\to 0} \phi(t) t^{\a-d}=0$.
\item [\rm (ii)] There exists a closed set $A$ in $\reald$ such that $\reald\setminus A$ is connected,     
\hbox{$m_\phi(A)=0$}, and $\|\vx^A\|=1$, for every $x\in \reald$.
\end{itemize} 
\item[\rm (2)] There exists a closed set $A$ in $\reald$ such that $\reald\setminus A$ is connected,
the Hausdorff dimension of $A$ is $d-\a$, and $\|\vx^A\|=1$, for every $x\in \reald$.
\end{itemize} 
\end{theorem} 

\begin{proof}    
If $E\subset\reald$ is Borel measurable and  not ($\a$-)polar, then  $m_{t^{d-\a}}(E)=\infty$ 
 (see \cite[Theorem III.3.14]{landkof}), and therefore its  Hausdorff dimension is at least $d-\a$. 
Hence  the implication (ii)\,$\Rightarrow$\,(i) in (1) follows as in the proof of Theorem~\ref{main}.

To prove the implication (i)\,$\Rightarrow$\,(ii), let $B$ be any locally finite union of pairwise disjoint closed balls $\ov B(z,r_z)$ which
is unavoidable. Such a set $B$ is easily obtained. Indeed,  by \cite[V.4.6]{BH},
\begin{equation*} 
   \kappa:=\inf\{R_1^{B(0,1)}(x)\colon x\in B(0,4)\} >0.
\end{equation*} 
Let $0<R_1<R_2<\dots$ be such that $\limn R_n=\infty$. If $\a\le 1$, we assume that, for some $\delta>0$,
$(1+\delta)R_n<R_{n+1}$, $n\in\nat$, and define
\begin{equation*} 
    B_n:=\{x\in \reald\colon R_n\le |x|\le (1+\delta) R_n\}, \qquad n\in\nat.
\end{equation*} 
If $\a>1$, we omit $\delta$, that is, we define $B_n:=\partial B(0,R_n)$. 
By Proposition \ref{riesz}, the union of all $B_n$ is unavoidable (in $\reald$).  Next we fix
$$
      0<\b_n \le \frac 14\min\{\dist(B_n, B_{n-1}\cup B_{n+1}), 1/n\}
$$
(with $B_0:=\emptyset$)
and choose a finite set $Z_n$ in $B_n$ such that the balls $\ov B(z,\b_n)$, $z\in Z_n$,  are pairwise disjoint
and the balls $B(z,4\b_n)$ cover $B_n$. If $x\in B_n$, then there exists $z\in Z_n$ such that $x\in B(z,4\b_n)$,
and hence $    R_1^{B(z,\b_n)}(x)\ge \kappa$, 
by translation and scaling invariance. By Proposition \ref{AB-unavoidable},2, the union $B$ of all
$B(z,\b_n)$, $z\in Z_n$, $n\in\nat$, is unavoidable. 

Now let $\phi$ be any measure function such that $\liminf_{t\to 0} \phi(t)t^{\a-d}=0$.
There exists a compact $F\subset B(0,1)$ of Cantor type (such that $B(0,1)\setminus F$ is connected), 
 which is not $(\a$-)polar, but satisfies $m_{\phi}(F)=0$
(see Theorem \ref{existence}).  Then, by \cite[VI.5.1 and  V.4.6]{BH},
\begin{equation*}
           \kappa':=\inf\{R_1^F(x)\colon x\in B(0,1)\}>0.
\end{equation*} 
Let $F_z:=z+r_z F$, $z\in Z$.  
 By translation and scaling invariance, for every $z\in Z$,
\begin{equation*} 
            R_1^{F_z}\ge \kappa' \on B(z,r_z).
\end{equation*} 
The union $A$ of all $F_z$, $z\in Z$, is closed in $\reald$ and $\reald\setminus A$ is connected.
Clearly, $m_\phi(A)=0$ and $A$ is unavoidable, by Proposition \ref{AB-unavoidable},2. 

Taking $ \phi(t):=t^{d-\a}(\log^+\frac 1t)\inv$ we obtain that the Hausdorff dimension of $A$
is at most $d-\a$, and hence equal to $d-\a$.
\end{proof}

\section{Champagne subsets of balayage spaces}\label{champagne-bal}

In this section we shall prove results on champagne subsets
of balayage spaces with small unavoidable unions of bubbles. 
These results will have immediate applications to various classes of harmonic
spaces and to non-local theories as, for example, Riesz potentials
on $\reald$ and censored stable processes on open sets
in $\reald$.

Let $(X,\W)$ be a balayage space such that points are polar
and the function $1$ is harmonic. Let $\rho$ be a metric on~$X$
 which is compatible with the topology of $X$. 
For every $x\in X$ and $r>0$, we define the open ball of center $x$
and radius $r$ by
$$
           B(x,r):=\{y\in X\colon \rho(x,y)<r\}. 
$$
We suppose that, for every compact $K$ in $X$, 
there exist $0<a\le 1$ and  $\ve>0$ such that
\begin{equation}\label{rtwor}
           R_1^{B(x,r)\setminus \ov B(x,r/2)}(x)\ge a,\qquad
\mbox{ for all }x\in K\mbox{ and }0<r< \ve.
\end{equation} 
Further, we  assume that we have a lower
semicontinuous  numerical function \hbox{$G>0$} on $ X\times X$,
finite and  continuous 
off the diagonal, and, for some $\rho_0>0$,
 a strictly increasing continuous function $\kap$ on $\left(0,\rho_0\right]$
with $\lim_{r\to 0}\kap(r)=0$  such that the following holds:
\begin{itemize}
\item [\rm (i)]
   For every $y\in X$,  $ G(\cdot,y)\mbox{ is a potential with superharmonic support } \{y\}$.   
\item [\rm (ii)]
For every $p\in \Px$, there exists a measure $\mu$ on $X$    such that 
\begin{equation}\label{p-rep}
p=G\mu:=\int G(\cdot,y)\,d\mu(y).
\end{equation} 
\item [\rm (iii)]
There exists a constant $c\ge 1$, 
such that, for every compact $K$ in $X$, there exists $0<\ve\le \rho_0$    
satisfying 
\begin{equation}\label{G-local}
     c\inv \le   G(x,y)\cdot \kap(\rho(x,y))  \le c, \qquad\mbox{ for all } x\in
     K\mbox{ and } y \in B(x,\ve).
\end{equation} 
\item [\rm (iv)] Doubling property: There exists a constant $C>1$
  such that, for all $0<r\le  \rho_0$,   
\begin{equation}\label{doubling}
\kap(r)\le C\kap(r/2).
\end{equation}
\end{itemize} 

\begin{remarks}\label{G-remarks}{\rm
1. If harmonic measures for relatively compact open sets $V$ are supported by $\partial V$, then, by the minimum principle,
(\ref{rtwor}) holds with $a=1$.

2. By (i), the measure $\mu$ in (ii) is supported by the superharmonic support  of         
$p$, that is, by the smallest closed set such that $p$ is harmonic  on its complement. 

3. Suppose that (i) holds and that there exists a measure $\mu_0$ on $X$ such that,     
 for some sub-Markov resolvent $\vvl$ on $X$ with proper potential kernel $V_0$, we have $V_0f=G(f\mu_0)$,
$f\in\B^+(X)$,  and the set of all $\mathbbm V$-excessive functions is $\W$ (see Section  \ref{sec-bal}, in particular,
Theorem \ref{bal-res}). 
Then, by \cite{maagli-87}, for every $p\in \Px$, there exists a~\emph{unique} measure $\mu$ such that 
(\ref{p-rep}) holds.

4. Of course, it is sufficient to have (\ref{G-local}) for a sequence 
$(K_n)$ of compact sets covering $X$ such that each $K_n$
is contained in the interior of $K_{n+1}$.  

5. If there exist $c\ge 1$ and $\g>0$ such that, for every compact $K$ in $X$,
there exists $\ve>0$ satisfying
\begin{equation*} 
      c\inv \rho(x,y)^{-\g} \le   G(x,y)  \le c \rho(x,y)^{-\g}, \qquad\mbox{ for all } x\in
     K\mbox{ and } y \in B(x,\ve),
\end{equation*} 
then (iii) and (iv) hold with $\kap (r):=r^\g$ and $C=2^\g$.

6. Suppose that  there exists  $c_0>0$ such that, for all ~$x,y,z\in X$,  
\begin{equation*} 
          G(x,y)\le c_0 G(y,x) \und \min\{G(x,y),G(y,z)\}\le c_0 G(x,z). 
\end{equation*} 
Then there \emph{exists} a metric $\rho$ on $X$ and constants $c,\g>0$  (see \cite[Proposition 14.5]{heinonen},
\cite[pp.\ 1209--1212]{convexity}, \cite{H-uniform})
such that $\rho$  is compatible with the topology of $X$ and
\begin{equation*} 
                                    c\inv \rho(x,y)^{-\g} \le G(x,y) \le c \rho(x,y)^{-\g},\qquad\mbox{  for all } x,y\in X.
\end{equation*} 
 }
\end{remarks}

\begin{examples}\label{examples}{\rm 
Taking Euclidean distance $\rho$ 
in (1) -- (4), and (6):

1. $X\ne \emptyset$ connected open set in $\reald$, classical potential theory:
 
a) $d\ge 3$:
   $c=1+\eta$,  $\eta>0$ arbitrary
   ($c=1$ if~$X=\reald$), $0<\rho_0< \infty$ arbitrary,
   $\kap(r)=r^{d-2}$, and $C=2^{d-2}$.

b) $d=2$,   $\real^2\setminus X$ nonpolar:
       $c=C=1+\eta$,  $\eta>0$ arbitrary, $\rho_0=2^{-1/\eta}$,
    $\kap(r)=(\log (1/r))\inv$. 

2. $X=\reald$, $d\ge 1$, $0<\a<\min\{2,d\}$,
   Riesz potentials (isotropic $\a$-stable processes):
   $c=1$, $\kap(r)=r^{d-\a}$, $C=2^{d-\a}$.

3. $X\ne\emptyset$ bounded $C^{1,1}$ open set  in $\reald$, $d\ge 2$,
   $\a\in (1,2)$, censored $\a$-stable process on $X$:
$\kap(r)=r^{d-\a}$, $C=2^{d-\a}$ (see Section \ref{section-censored}).

4. Of course, harmonic spaces given by (locally) uniformly elliptic partial differential
operators of second order on open sets in $\reald$ are covered as well (having local comparison
of the corresponding Green functions with the classical one).

5. Examples, where the underlying topological space is still some $\real^m$,
but the metric is  no longer the Euclidean metric, are given by sublaplacians
on stratified Lie algebras (see \cite[Theorem 1.1]{han-hue-subl}, where, by \cite[Proposition 14.5]{heinonen},
 the quasi-metric $d_N$ is equivalent to a power of a metric).
Special cases    for such sublaplacians are the 
Laplace-Kohn operators on  Heisenberg groups $\mathbbm H_n$, $n\in\nat$
(see \cite[VIII.5.7]{BH}).

6. Finally, we note that, more generally than in  our second standard example,  
our assumptions are satisfied by  isotropic unimodal L\'evy processes
on $\reald$, $d\ge 3$, having a lower scaling property for the characteristic function $\psi$
(see Section \ref{levy}).
}
\end{examples}

Aiming at the result in  Theorem \ref{unavoidable-bal} we claim the following.  

\begin{theorem}\label{KZ}
Let $0<\kappa< (cC)^{-4}$,
$\eta\in (0,1)$, and $h\colon (0,1)\to (0,1)$ satisfying  
$\lim_{t\to 0} h(t)=0$.
Further, let  $K\ne \emptyset $ be a  compact in $X$
 and let $K'$ be a
compact neighborhood of $K$. 

Then there exist a~finite set $Z$ in $K'$ 
{\rm(}even in $K$, if $K$ is not thin at any of its points{\rm)}
and radii $0<r_z<\min\{\eta,\rho_0\}$, $z\in Z$, such that the following holds.
\begin{itemize}
\item The closed balls $\ov B(z,r_z)$, $z\in Z$,  are
 pairwise disjoint subsets  of $K'$.
\item The union $E$ of all $\ov B(z,r_z)$, $z\in Z$, satisfies
  $\|\vx^E\|\ge \kappa$, for every $x\in K$.\footnote{Let us
 note that, for $d=2$ in
 Example \ref{examples}.1, $\kappa$ may be as close to $1$ as we want.}
\item The sum $\sum_{z\in Z} \kap(r_z)\, h(r_z)$ is strictly smaller than
  $\eta$.
\end{itemize} 
\end{theorem} 

 Essentially, the idea for our proof 
is the following. Let $\mu$ denote the equilibrium measure for $K$.
For $\b>0$, we consider a partition of $K$ into finitely 
many Borel measurable sets $K_z$, $z\in Z$, such that  
$K\cap B(z,\b/3)\subset K_z\subset K\cap B(z,\b)$  
 and choose $0<r_z < \b/3$    
such that $\kap(r_z)$ is approximately $\mu(K_z)$ (possible if $\b$
is small).
Then the closed balls $\ov B(z,r_z)$ are pairwise 
disjoint and the sum $\sum_{z\in Z} \kap(r_z)h(r_z)$ is bounded by
a multiple of $\mu(K)\max_{z\in Z}h(r_z)$, which is smaller than $\eta$ provided
$\b$~is sufficiently small. 
For every $z\in Z$, we define
$\nu_z:=\mu(K_z)\|\mu_z\|\inv \mu_z$, where $\mu_z$ is the equilibrium measure
of $B(z,r_z)$. Let $\nu$ denote the sum of all $\nu_z$, $z\in Z$.
If we can show that $c_1 G\nu \le G\mu\le C_1 G\nu$ (which will require
some effort), we may conclude that
$$
        R_1^E\ge c_1 G\nu\ge c_1C_1\inv G\mu,
$$
where $G\mu=1$ on $K$, if $K$ is not thin at any of its points.

Since $K$ may not have this property and since we do not know
if the measures~$\mu_z$ have enough mass near the boundary of $B(z,r_z)$
(no problem,  if harmonic measures for open sets $V$ are supported by $\partial V$), we have to proceed in a 
more subtle  way.

In a way, our approach resembles to what has been done in \cite{aikawa-borichev}
to obtain a~result on  quasi-additivity of capacities.  
In \cite{aikawa-borichev}, however,  
given equilibrium measures on well separated small sets are spread out on larger balls to obtain 
a one-sided estimate between the corresponding potentials, whereas we cut a measure, given 
on a large set, into pieces, which are concentrated on small balls and lead to a two-sided
estimate.

\begin{proof}[Proof of Theorem \ref{KZ}]
We fix
$0<\ve<\dist(K,X\setminus K')/3$ and $0<a\le 1$ 
such that (\ref{rtwor}) and (\ref{G-local}) hold 
with $K'$ instead  of $K$.
If $K$ is not thin at any of its points,  let $\vp:=1_K$. If not, we choose $\vp\in\C(X)$ 
such that $1_K\le  \vp\le 1$ and the support of 
$\vp$ is contained in the $\ve$-neighborhood of $K$. 
Then $R_\vp$ is a continuous potential
which is equal to $1$ on $K$ and harmonic outside the support of $\vp$.
By (\ref{p-rep}), there exists a~measure~$\mu$ on~$X$ such that     
$G\mu=R_\vp$. The support $L$ of $\mu$ is contained
in the $\ve$-neighborhood of $K$; it is even a~subset of~$K$,
if $K$~is not thin at any of its points.

There exists  $\g>0$ such that 
\begin{equation}\label{def-g}
\kappa\le \frac {1-\g}{c^2C^2(\g cC+c^2C^2)}\, .
\end{equation}  
Let $\tau:=\inf_{x\in K'} G\mu(x)$. Then $0<\tau\le 1$ and  there exists
$0<R<\ve/3$ such that 
\begin{equation}\label{kato}
                 \sup\nolimits_{0<t<R}   h(t)\le  \frac
                              {a\g\tau \eta}{\mu(L)}
 \und            
  \sup\nolimits_{x\in L} G(1_{B(x,R)}\mu)\le \tau':=
                     \frac {a\g\tau} {c\, C^{3}},
\end{equation} 
where the second inequality follows from the fact that
$G\mu\in\C_b(X)$, and hence $\mu$
does not charge points (which are polar).
(Indeed, for every $x\in L$, there exists $0<s_x<\ve/3$
such that $G(1_{B(x,2s_x)}\mu)\le \tau'$. There exist $x_1,\dots,x_m\in
L$ such that the balls $B(x_1,s_{x_1}),\dots, B(x_m,s_{x_m})$
cover $L$.
Let $R:=\min\{s_{x_1},\dots,s_{x_m}\}$.
Given $x\in L$, there exists $1\le j\le m$ such that $x\in
B(x_j,s_{x_j})$, hence $B(x,R)\subset B(x_j,2s_{x_j})$, 
and therefore $G(1_{B(x,R)}\mu)\le \tau'$.)

Since, by assumption, $G>0$ and $G$ is continuous off the diagonal, there exists
$\delta>0$ such that,  for all $x\in K'$,
\begin{equation}\label{G-global'}
       G(x,y)\le C G(x,y'),\quad\mbox{ if } y,y'\in
       K'\setminus B(x,R)\mbox{ and }\rho(y,y')<\delta.
\end{equation} 
Finally, let 
$$
\b:=(1/2)\min\{\delta, R, \eta, \rho_0\}. 
$$
 There exists a finite set $Z$ in $L$ such that the
balls $B(z,\b/3)$, $z\in Z$, are pairwise disjoint and the balls 
$B(z,\b)$  cover $L$. There exists a partition of $L$ into
Borel measurable sets $L_z$, $z\in Z$, such that 
\begin{equation}\label{kz}
             L\cap B(z,\b/3)\subset L_z \subset L\cap B(z,\b).
            \end{equation} 
Indeed, let $Z=\{z_1,\dots,z_M\}$ and, for every $1\le j\le M$,
let $P_j':=L\cap B(z_j,\b/3)$, $P_j'':=L\cap B(z_j,\b)$, and let $L'$
be the union of the pairwise disjoint sets $P_1',\dots,P_M'$.
 We recursively define $L_{z_1},\dots, L_{z_M}$ by $L_{z_1}:=P_1'\cup
 (P_1''\setminus L')$ and
$$
       L_{z_j}:=(P_j'\cup (P_j''\setminus L'))\setminus 
(L_{z_1}\cup\dots\cup L_{z_{j-1}}), \qquad 1<j\le M.
$$

For the moment, let us fix $z\in Z$. Since 
$1\le c G(z,\cdot)\kap(\b)$ on $B(z,\b)$ by (\ref{G-local}), 
we see, by  (\ref{kato}) and (\ref{doubling}), that
$$
       (a\g\tau)\inv  \mu(B(z,\b))\le (a\g\tau)\inv c\, \kap(\b)
      G(1_{B(z,\b)}\mu)(z) \le C^{-3} \kap  (\b)\le \kap(\b/8). 
$$
So $(a\g\tau)\inv \mu(L_z)<\kap(\b/8)$, by (\ref{kz}), and hence
there exists (a unique)
\begin{equation}\label{def-rz}
    0<r_z<\b/8 \quad\mbox{ such that } \quad
             \kap(r_z)=(a\g\tau)\inv \mu(L_z).
\end{equation} 
In particular, by (\ref{kato}),
$$
         \sum\nolimits_{z\in Z} \kap(r_z) \, h(r_z)
         < \frac{\eta}{\mu(L)}\,  \sum\nolimits_{z\in Z} \mu(L_z) =\eta.
$$

By (\ref{rtwor}), for every $z\in Z$, there exists 
$\vp_z\in \K^+(X)$ such that $\vp_z\le 1_{B(z,r_z)\setminus \ov B(z,r_z/2)}$
and $R_{\vp_z}(z)>a/2$. Since $R_{\vp_z}\in \Px$ and $R_{\vp_z}$ is harmonic outside the compact 
$A_z:=\ov B(z,r_z)\setminus B(z,r_z/2)$, 
there exists a~measure $\mu_z$ on $A_z$  such that   
\begin{equation}\label{mz}
                G\mu_z=R_{\vp_z}.
\end{equation} 
For every $y\in A_z$,
$G(z,y) \kap(r_z) \le C G(z,y)\kap (r_z/2) \le cC$,
and hence,  by~(\ref{def-rz}),   
$$
\frac {\mu(L_z)} {2\g\tau}=\frac {a\kap(r_z)}2< G\mu_z(z) \kap(r_z)= \int
G(z,y)\kap(r_z)\,d\mu_z(y)\le c C \|\mu_z\|.
$$
Let 
\begin{equation}\label{def-nz}
      \nu_z:=\mu(L_z)\,\|\mu_z\|\inv \mu_z.
\end{equation} 
Then $G\nu_z\le 2\g\tau cC G\mu_z\le 2\g\tau cC\le 2\g cC G\mu$ on $K'$.
By the minimum principle, 
\begin{equation}\label{nmG}
 G\nu_z\le  2\g c C G\mu.
\end{equation} 
 Defining $\nu:=\sum_{z\in Z} \nu_z$
we claim that 
\begin{equation}\label{munu}
(2\g c C+c^2C^2)\inv   G\nu \le   G\mu \le (1-\g)\inv c^{2}C^2 G\nu.
\end{equation} 
Having (\ref{munu}), the proof of the proposition is easily finished.
Indeed, the measure $\nu$ is supported by the union $E$
of all closed balls $\ov B(z,r_z)$, $z\in Z$, and $G\nu$
is continuous. Since $G\mu\le 1$, the first inequality  and the minimum principle hence yield that 
\begin{equation}\label{nu-E}
     (2\g cC+ c^2C^2)\inv    G\nu\le R_1^E. 
\end{equation} 
Finally, (\ref{def-g}),  the second inequality of (\ref{munu}), and (\ref{nu-E})  imply that 
$$
  \kappa  1_K  \le  \kappa G\mu\le \frac{1-\g}{c^2C^2  (2\g cC+ c^2C^2)}\, G\mu
\le R_1^E.
$$

So it remains to prove that (\ref{munu}) holds. 
To that end,  let us fix $z\in Z$ and define
$$
       V:=B(z,R), \qquad \mu':=1_{L\setminus V}\mu.
$$
By (\ref{kato}) and the minimum principle, $G(1_V\mu)\le \g  G\mu$, and hence
\begin{equation}\label{mu'}
          G\mu'\ge  (1-\g) G\mu.
\end{equation} 
Let  $x\in B(z,\b)$, $z'\in Z\setminus \{z\}$, and  $y,y'\in L_{z'}$.
Then $\rho(y,y')<2\b\le\delta$. 
If $L_{z'}\cap B(x,R)=\emptyset$,  we hence conclude, by
(\ref{G-global'}), that
$$
       G(x,y)\le C G(x,y').
$$
Let us suppose now that  $L_{z'}\cap B(x,R)\ne \emptyset$. Then
$\max\{\rho(x,y),\rho(x,y')\}< R+2\b$, hence $y,y'\in B(x,\ve)$.
If $y\in L_{z'}\setminus V$, then $\rho(x,y)\ge \rho(z,y)-\rho(x,z)\ge R-\b\ge \b$,
and hence $\rho(x,y')< \rho(x,y)+\rho(y,y')< 3\rho(x,y)$. 
If $x\in \ov B(z,r_z)$ and $y\in \ov B(z',r_{z'})$, then  $\rho(x,y')\le 4 \rho(x,y)$, since $\rho(x,z')\ge (\frac 23-\frac 18)\b=
\frac {13}{24}\b$,
$ \rho(x,y')\le \rho(x,z')+ \b\le \frac{37}{13}\rho(x,z')$, and $\rho(x,y)\ge \rho(x,z')-\frac 18\b\ge \frac {10}{13}\rho(x,z')$.
By the monotonicity of $\kap$ and (\ref{doubling}),  in~both cases
$$
\kap(\rho(x,y'))\le \kap(4\rho(x,y))\le C^2\kap(\rho(x,y)),
$$
 and therefore, by (\ref{G-local}) (with $K'$ in place of~$K$), 
$$
    G(x,y)\le c \bigl(\kap(\rho(x,y))\bigr)\inv
      \le c C^2\bigl(\kap(\rho(x,y'))\bigr)\inv
    \le c^2C^2 G(x,y').
$$

By integration, we conclude that 
\begin{equation}\label{z-mu}
G(1_{L_{z'}\setminus V}\mu)\le c^2C^2\,\frac
{\mu(L_{z'}\setminus V)}{\|\mu_{z'}\|}\, G\mu_{z'} \le c^2C^2
G\nu_{z'}  \on B(z,\b),
\end{equation} 
and
\begin{equation}\label{z-nu}
G\nu_{z'}\le c^2C^2 G(1_{L_{z'}}\mu)\on \ov B(z,r_z).
\end{equation} 
Summing (\ref{z-mu}) over all $z'\in Z\setminus \{z\}$ we obtain that $G\mu'\le c^2C^2 G\nu$ on $B(z,\b)$
and hence, by (\ref{mu'}), $G\mu\le (1-\g )\inv c^2C^2 G\nu$ on $B(z,\b)$. Since the balls $B(z,\b)$, $z\in Z$,
cover the support $L$ of $\mu$, an application of the the minimum principle yields the second inequality of~(\ref{munu}).
Summing (\ref{z-nu}) over all $z'\in Z\setminus \{z\}$ and using (\ref{nmG}), we see that
$ G\nu\le 2\g cC G\mu+c^2C^2 G\mu'\le( 2\g cC +c^2C^2)G\mu$ on $\ov B(z,r_z)$. By the minimum
principle, the second inequality of  (\ref{munu}) follows.

\end{proof}

In the classical case (see Example \ref{examples}.1), Theorem \ref{KZ}
implies an improved version of \cite[Theorem 1.1]{HN-champagne} (recalled in this paper also   
as  Theorem  \ref{main-old}) 
in the (most natural) case, where $\lim_{t\to 0} h(t)=0$.

\begin{corollary}\label{corollary-classic}
Let $U$ be a nonempty connected open set in $\reald$, $d\ge 2$,
such that $\reald\setminus U$ is nonpolar if $d=2$.
Let $(V_n)$ be an exhaustion of $U$ by relatively compact 
open sets $V_n$, $n\in\nat$, such that $\partial V_n$ is not thin at any of its points.
Finally, suppose that  $h\colon (0,1)\to (0,1)$ satisfies 
\hbox{$\lim_{t\to 0} h(t)=0$}, and let
$\psi\in \C(U)$ such that $0<\psi\le \dist(\cdot,U^c)$.  

Then, for every $\delta>0$,  there exist finite sets $Z_n$
 in $\partial V_n$ and
$0<r_z<\psi(z)$, $z\in Z_n$, $n\in\nat$,  such that 
for the union $Z$ of all $Z_n$ 
and the union $B$ of all $\ov B(z,r_z)$, $z\in Z$, the following holds:
\begin{itemize} 
\item $U\setminus B$ is a~champagne subdomain of~$U$ and $B$
is unavoidable in   $U$.
\item $\sum\nolimits_{z\in Z}  \kap (r_z)h(r_z)<\delta$.
\end{itemize} 
\end{corollary}

\begin{proof}
 By Proposition \ref{harmonic}, the union $A$ of all boundaries
  $\partial V_n$, \hbox{$n\in\nat$}, is unavoidable in~$U$.
Let $\delta\in (0,1)$ and   
$$
\eta_n:=(1/2) \min\{2^{-n}\delta, 
\dist(\partial V_n, \partial V_{n-1}\cup \partial V_{n+1}),
       \inf \psi(\partial V_n)\},\qquad n\in\nat
$$
(with $V_0:=\emptyset$).

By Theorem  \ref{KZ}, we may choose $\kappa>0$
(by Example \ref{examples}.1, any  $0<\kappa<2^{-4(d-2)}$ will do),
finite sets $Z_n$ in $\partial V_n$, 
and $0<r_z<\eta_n$,  $z\in Z_n$, $n\in\nat$,  such that
$$
\sum\nolimits_{z\in Z_n} \kap(r_z) h(r_z)<\eta_n,
$$
 the closed balls 
$\ov B(z,r_z)$, $z\in Z_n$, are pairwise disjoint
 and    their union $E_n$ satisfies 
$$
\|\U \vx^{E_n}\|\ge \kappa, \qquad x\in \partial V_n.
$$
Let 
$$
Z:=\bigcup\nolimits_{n\in\nat}Z_n \und 
      B:=\bigcup\nolimits_{z\in Z}\ov B(z,r_z)=\bigcup\nolimits_{n\in\nat} E_n.
$$
Clearly, $U\setminus B$ is a champagne subdomain of $U$
and
$$
\sum\nolimits_{z\in Z}\kap(r_z)h(r_z)<\sum\nolimits_{n\in\nat}2^{-n}\delta=\delta.
$$
If $x\in A$, then $x\in \partial V_n$ for some  $n\in\nat$, and  hence,
$ \|\U\vx^B\|\ge \|\U\vx^{E_n}\|\ge \kappa$.
So $B$ is unavoidable, by Proposition \ref{AB-unavoidable},2, that is,
$\|\U\vx^B\|=1$, for every $x\in U$.
\end{proof} 

Let us return to the general situation we were considering before
this application of Theorem \ref{KZ}. 

\begin{theorem}\label{unavoidable-bal}
Let $h\colon (0,1)\to (0,1)$ be such that $\lim_{t\to 0} h(t)=0$, let $\delta>0$
and $\psi\in \C(X)$, $\psi>0$. 

Then  there exist a locally finite set $Z$ in $X$ and \hbox{$0<r_z<\psi(z)$},  
\hbox{$z\in Z$},  
such that the closed balls $\ov B(z,r_z)$ are pairwise disjoint,
the union  of all $\ov B(z,r_z)$ is unavoidable, and 
$\sum_{z\in Z} \kap(r_z) h(r_z)<\delta$.
\end{theorem} 

\begin{proof} Let us choose an exhaustion $(U_m)$ of $X$. 
By Proposition \ref{shells}, there exist $m_n\in \nat$, $n\in\nat$,
such that $m_{n}+4\le m_{n+1}$ and the union $A$ of 
the  compact ''shells'' $K_n:=\ov U_{m_{n+1}}\setminus U_{m_n+4}$
is unavoidable. For every $n\in\nat$, the compact 
$$
K_n':=\ov U_{m_{n+1}+1}\setminus U_{m_n+3}
$$ 
is a neighborhood of
$K_n$, and the sets $K_n'$, $n\in\nat$, are pairwise disjoint.
Assuming without loss of generality that $\delta<1$ we define 
$$
\eta_n:=\min\{2^{-n}\delta, \inf \psi(K_n')\}.
$$
Let $0<\kappa<(cC)^{-4}$.
By Theorem \ref{KZ}, there are  finite sets $Z_n$ in $K_n'$
and $0<r_z<\eta_n$, $n\in\nat$, such that
$$
\sum\nolimits_{z\in Z_n} \kap(r_z) h(r_z)<\eta_n,
$$
 the closed balls 
$\ov B(z,r_z)$, $z\in Z_n$, are pairwise disjoint,  contained
in $K_n'$, and    their union $E_n$ satisfies 
$$
\|\vx^{E_n}\|\ge \kappa, \qquad x\in K_n.
$$
The proof is finished in a similar way as 
the proof of Corollary \ref{corollary-classic}.
\end{proof} 

Applying Theorem \ref{unavoidable-bal}  to Example \ref{examples}.2
we  obtain the following.

\begin{corollary}\label{a-stable}
Let $X=\reald$, $d\ge 1$, $0<\a<\min\{d,2\}$, and let $\W$ be the
convex cone of all increasing limits of Riesz potentials $G\mu\colon
x\mapsto \int |x-y|^{\a-d}\,d\mu(y)$ {\rm(}$\mu$~finite measure on
$\reald$ with compact support\,{\rm)}.
Let $h\colon (0,1)\to (0,1)$ be such that $\lim_{t\to 0} h(t)=0$,
and let $\psi\in \C(X)$, $\psi>0$. 

Then, for every $\delta>0$,
there is a locally finite set $Z$ in $X$ and \hbox{$0<r_z<\psi(z)$},
\hbox{$z\in Z$},  
such that the closed balls $\ov B(z,r_z)$ are pairwise disjoint,
the union  of all~$\ov B(z,r_z)$ is unavoidable, and 
$\sum_{z\in Z} r_z^{d-\a} h(r_z)<\delta$.
\end{corollary} 

\begin{remark}{\rm    
Let $\a\in (1,2)$ and $0<R_1<R_2<\dots$ such that $\limn R_n=\infty$. Using Proposition \ref{riesz} it is easy to see that
(as in the classical case) we may choose  $Z=\bigcup_{n\in\nat}  Z_n$, where  $Z_n\subset \partial B(0,R_n)$ 
and $r_z$ is the same for all $z\in Z_n$.
}
\end{remark}

\section{Application to isotropic unimodal semigroups}\label{levy}

To cover more general isotropic unimodal L\'evy processes, 
we study  
 (in a purely analytic way) the following  isotropic situation.

Let $\mathbbm P=(P_t)_{t>0} $ be a right continuous sub-Markov semigroup on $\reald$, $d\ge 1$, and let
$V_0$ denote the potential kernel of $\mathbbm P$ (see Section \ref{sec-bal}): 
$$
                V_0f(x)=\int_0^\infty P_tf(x)\,dt, \qquad f\in \B^+(\reald),\ x\in\reald.
$$

Further, let $g$ be a decreasing  numerical function on $\left[0,\infty\right)$ such that  
$0<g<\infty$ on $(0,\infty)$, $\lim_{r\to 0} g(r)=g(0)=\infty$, $\lim_{r\to \infty} g(r)=0$, and 
\begin{equation}\label{g-ass}
\int_0^1 g(r) r^{d-1}\,dr<\infty
\end{equation} 
(later on we shall replace (\ref{g-ass}) by the stronger property (\ref{g-decay})). 
We suppose that 
\begin{equation}\label{ass-1}
                                                   G_0:=g(|\cdot|)\in \es  \und  V_0f(x)=G_0\ast f(x)=\int G_0(x-y) f(y)\,dy,
\end{equation} 
for all $f\in \B^+(\reald)$ and $x\in \reald$ (where $\es$ is the set of all $\mathbbm P$-excessive functions).

\begin{remarks} {\rm 1}. 
Of course, the assumptions, including {\rm (\ref{g-decay})},  are satisfied in the classical case      
and by Riesz potentials with $g(r)=r^{\a-d}$, $\a\in (0,2)$, $\a<d$. 

{\rm 2}. 
 Let us note that {\rm (\ref{ass-1})} holds with $G_0=\int_0^\infty p_t\,dt$, 
 if there exists a~Borel measurable 
function  $(t,x)\mapsto p_t(x)$ on $(0,\infty)\times \reald$ such that each $p_t$ is radial and decreasing 
{\rm(}that is, $p_t(x)\le p_t(y)$ if $|y|\le  |x|${\rm)}, $p_s\ast p_t=p_{s+t}$, for all $s,t>0$, and $P_tf=p_t\ast f$,
for every $f\in \B^+(\reald)$. 

{\rm 3.} 
 In particular, our hypotheses, including {\rm (\ref{g-decay})},  are satisfied by  the transition semigroups of
the isotropic unimodal L\'evy processes $\mathfrak X=(X_t, P^x)$ studied   
in~{\rm \hbox{\cite{grzywny, mimica-vondracek}}} {\rm(}see also {\rm\cite{hawkes}}{\rm)}.  
It is assumed that the characteristic function~$\psi$ for such a process~$\mathfrak X$ {\rm(}given by 
$  e^{-t\psi(x)}=   E^0 [e^{i\langle x,X_t\rangle}]$,  $t>0$, $x\in\reald${\rm)} satisfies a \emph{weak lower scaling 
condition}: There exist  $\a>0$,  $0\le C_L\le 1$, and $R_0>0$ such that 
\begin{equation}\label{wlsc}
\psi(\lambda x)\ge C_L\lambda^\a \psi(x),\qquad \mbox{for all } \lambda\ge 1\mbox{  and }  x\in B(0,R_0)^c
\end{equation} 
{\rm(}see {\rm\cite[p.~2]{grzywny}}  and {\rm \cite[(1.4)]{mimica-vondracek})}.  Having shown that $g(r)\approx r^{-d} \psi(1/r)\inv$
{\rm(}see {\rm \cite[Proposition 1, Theorem 3]{grzywny}} or {\rm \cite[Lemma 2.1]{mimica-vondracek})},  
condition {\rm (\ref{wlsc})}
implies that {\rm (\ref{g-wlsc})}  holds {\rm(}which in turn leads to  {\rm (\ref{g-decay})}{\rm )}. 

 Examples in the case $d\ge 3$ are listed in  {\rm\cite[p.~3]{mimica-vondracek}};
for $d\le 2$, see {\rm\cite[Section 6]{mimica-vondracek}}.

{\rm4.}
For the general possibility of constructing new examples by subordination see Theorem \ref{subordination}.    
\end{remarks}

\begin{lemma}\label{VC}
For every bounded $f\in \B^+(\reald)$ having  compact support, the function~$V_0f$ is contained in $\es\cap C(\reald)$ and  vanishes at infinity.
\end{lemma} 

\begin{proof} Since $V_0(\B^+(\reald))\subset \es$, 
the statement follows immediately from  (\ref{g-ass}), (\ref{ass-1}),  and  our transience property $\lim_{r\to \infty} g(r)=0$.
\end{proof} 

The next result  as well as Theorem \ref{green-rep} is of interest in its own right.  

\begin{theorem} \label{bal-true}
\begin{enumerate}
\item[\rm 1.]
 $(\reald,\es)$ is a balayage space. 
\item[\rm 2.]
Every point in $\reald$ is polar.
\item[\rm 3.] 
Borel measurable finely open $U\ne\emptyset$ have strictly positive Lebesgue measure.
\item[\rm 4.] 
If \,$\mathbbm P$ is a Markov semigroup, then the constant $1$ is harmonic.
\end{enumerate}
\end{theorem} 

\begin{proof} 1. Consequence of Lemma \ref{VC} and  Theorems \ref{II.4.7} and \ref{V-approx}. 

2. Since $G_0\in \es$, the origin is polar. By translation invariance, every point in~$\reald$ is polar.

3. Let $\mathbbm V=(V_\lambda)_{\lambda>0}$ be the resolvent of $\mathbbm P$ and let 
$U$ be a Borel measurable finely open set   and $x\in U$. 
 By Theorem \ref{II.4.7}, there exists
$\lambda>0$ such that $V_\lambda(x,U)>0$. The proof is finished, since $V_0(x,\cdot)\ge V_\lambda(x,\cdot)$
and $V_0(x,\cdot)$ is absolutely continuous with respect to Lebesgue measure
(having the density $y\mapsto G(x-y)$). 

4. True, by \cite[III.7.6]{BH}. 
\end{proof}

The following proposition  will be useful for us (and shows that any open set satisfying an exterior cone condition
is regular for the Dirichlet problem):

\begin{proposition}\label{cone}
Let $z,z_0\in\reald$, $z\ne z_0$  and $0<r<|z-z_0|$.
Then the open set $U_0:=\convex(\{z\}\cup B(z_0,r))\setminus \{z\}$
is not thin at $z$, that is, $z$ is contained in the fine closure of $U$.\footnote{The fine topology
 is the coarsest topology for which all functions in $\W$ are continuous,
and $\convex (A)$ denotes the convex hull of $A$.}
\end{proposition} 

\begin{proof} We may assume without loss of generality that $z=0$. Let $R:=|z_0|$. There exist  
 $z_1,\dots, z_m\in \partial B(0,R)$ such that  the balls $B(z_j,r)$, $0\le j\le m$, cover~$\partial B(0,R)$. 
Then $B(0,R)\setminus \{0\} $ is covered by the union of the sets
 $U_j:=\convex(\{0\}\cup B(z_j,r))\setminus \{0\}$, $0\le j\le m$.  
 Since the origin is polar,  it is contained 
in the fine closure of~$B(0,R)\setminus \{0\}$, hence in the fine closure of one of these sets. By radial invariance, 
the origin is contained  in the fine closure of $U_0$.  
\end{proof} 

\begin{corollary}\label{ball-reg}
Every open ball $B(x,r)$, $x\in\reald$, $r>0$, is finely dense in the closed ball $\ov B(x,r)$.     

In particular, for all $Z\subset \reald$ and $r_z>0$, $z\in Z$, the union of all $B(z,r_z)$
is unavoidable if and only if the union of all $\ov B(z,r_z)$ is unavoidable.
\end{corollary} 

\begin{proof} Let $x\in\reald$, $r>0$, $z\in \partial B(x,r)$. Then \hbox{$\convex(\{z\}\cup B(x,r/2))\setminus \{z\}\subset B(x,r)$}.
Thus,  by Proposition \ref{cone},  the point $z$ is contained in the fine  closure of $B(x,r)$. 
\end{proof} 

For all $x,y\in \reald$, let    
$$
        G(x,y):= G_0(x-y).
$$
Then $G$ is symmetric (that is, $G(x,y)=G(y,x)$, $x,y\in\reald$), continuous outside the diagonal,  and it is 
a  Green function for $(\reald,\es)$:

\begin{theorem}\label{green-rep}    
\begin{itemize} 
\item[\rm 1.]                        
For every $y\in\reald$, $G_y:=G(\cdot,y)$     
is a potential with superharmonic support $\{y\}$.
\item[\rm 2.]           
Let $\mu$ be a~measure on $\reald$.   
Then $G\mu:=\int G_y\,d\mu(y)\in \es$ and,
provided $G\mu$ is a~potential,\footnote{For the definition of potentials on balayage spaces see Section 2.}
the support of~$\mu$ is the superharmonic support of~$G\mu$. 
\item[\rm 3.] 
For every potential $p$ on $\reald$, there exists a~{\rm(}unique{\rm)} measure $\mu$ on $\reald$
such that $p=G\mu$. 
\end{itemize}
\end{theorem} 

\begin{proof} Having (1) we shall obtain (2) and (3) from (\ref{ass-1}) and  \cite[Lemma 2.1 and Theorem 4.1]{HN-rep-potential}, 
since using Lemma~\ref{VC} we may construct $f\in \C^+(\reald)$, $f>0$,  such that $V_0f\in \C(\reald)$. 
 We note that (3) also follows from \cite{maagli-87}. 

To prove (1) we may assume without loss of generality that $y=0$.   We   know     
that $G_0\in\es$ is not only lower semicontinuous, but also radial and decreasing. 
Therefore, by Proposition  \ref{cone}, $G_0$ is continuous on $\reald\setminus \{0\}$.
Indeed, let $z\in\reald$. 
Then 
$$
\g:=\inf\{ G_0(x):|x|<|z|\}=\limsup\nolimits_{x\to z} G_0(x).
$$
By Corollary \ref{ball-reg}, $G_0(z)\ge   \g$, since $G_0$ is finely continuous.   
So $G_0$ is continuous at~$z$.  
Moreover, $G_0$ vanishes at infinity, since $\lim_{r\to\infty} g(r)=0$.  
Hence $G_0$ is a potential. 

To see that $G_0$ is harmonic on $\reald\setminus \{0\}$, let us fix a bounded open
set~$U\ne\emptyset$ in~$\reald$  and a bounded open neighborhood~W
of~$\ov U$ such that $0\notin \ov W$. We define 
$$
       r:=\frac 13\min\{\dist(\ov U,W^c), \dist (0,\ov W)\}, \qquad v:=V_01_{B(0,r)}=\int_{B(0,r)}  G_y \,dy,
$$
and fix $x\in U$. Since $\ov W\cap \ov B(0,r)=\emptyset$, we know that
\begin{equation}\label{vwv}
              \vx^\wc(v)   =v(x)
\end{equation} 
(immediate consequence of \cite[II.7.1]{BH} or, probabilistically, from the strong Markov property 
for a~corresponding Hunt process). 
Further, for every $y\in B(0,r)$, $G_y\in \es$ and $y+U\subset W$, hence 
$\vx^\wc(G_y)\le  \vx^{(y+U)^c}(G_y)\le  G_y(x)$. 
So we obtain that
\begin{equation*} 
    \vx^\wc(v)=\int_{B(0,r)} \vx^\wc(G_y)\,dy\le \int_{B(0,r)} \vx^{(y+U)^c}(G_y)\,dy\le \int_{B(0,r)} G_y(x)\,dy=v(x).
\end{equation*} 
Having (\ref{vwv}) we see that $\vx^{(y+U)^c}(G_y)=  G_y(x)$  for almost every $ y\in B(0,r)$, where, by translation
invariance, $\vx^{(y+U)^c}(G_y)=\ve_{x-y}^\uc(G_0)=\hat R_{G_0}^\uc(x-y)$. So, for almost every $ y\in B(0,r)$,
\begin{equation}\label{Gxy}
                     G_0(x-y)=\hat R_{G_0}^\uc(x-y).
\end{equation} 
By fine continuity and Theorem \ref{bal-true},3, (\ref{Gxy}) holds for \emph{every}  $y\in B(0,r)$. 
In particular, $G_0(x)=\hat R_{G_0}^\uc(x)=\vx^\uc(G_0)$. This finishes the proof. 
\end{proof}

To get property (\ref{G-local}) (even with $c=1$) it suffices to  define
$$
                \kap(r):=g(r)\inv, \qquad r>0.
$$

For every ball $B$ let $|B|$ denote the Lebesgue measure of $B$   
 and let $\lambda_B$ denote normalized Lebesgue measure on $B$
(the measure on $B$ having density $1/|B|$ with respect to Lebesgue measure). 
From now on, let us replace (\ref{g-ass}) by the following stronger hypothesis.

\begin{ass} 
There exist $C_G\ge 1$ and $0<r_0\le \infty$ such that, for every $0<r< r_0$,   
\begin{equation}\label{g-decay}
       d     \int_0^r  s^{d-1}g(s) \,ds \le C_G \, r^d g(r) 
\end{equation} 
or, equivalently,    
\begin{equation}\label{vg}
       G\lambda_{B(0,r)}(0) = \frac 1{ |B(0,r)|}        \int_{B(0,r)} G_y(0)\, dy \le C _G \,g(r).
\end{equation}                    
\end{ass}

Let us  note   that (\ref{g-decay}) holds with    constant     
 $C_G=(d/\a) C$, if $g$ has  the following \emph{decay property}:  There exists $C>0$   such that  
\begin{equation}\label{g-wlsc}
 g(\g r)\le C \g^{\a-d}  g(r), \qquad \mbox{ for all  } 0<\g<1  \mbox{ and } 0<r< r_0.
\end{equation} 
Indeed, if (\ref{g-wlsc}) holds and $0<r< r_0$, then
$$
       \int_0^r s^{d-1} g(s)\,ds= r^d\int_0^1 \g^{d-1} g(\g r)\,d\g
\le   C r^d g(r) \int_0^1 \g^{d-1}\g^{\a-d}\, d\g=\a\inv   C r^d g(r)
 $$
(of course, the argument shows that we still get  (\ref{g-decay}), if $\g^{\a-d}$ in (\ref{g-wlsc}) is replaced by any $f(\g)\ge 0$
with $\int_0^1 \g^{d-1} f(\g)\,d\g<\infty$).

Moreover, we observe that, by  (\ref{vg}),      
\begin{equation}\label{vg-global}
             G\lam_{B(0,r)} (x)\le C_G g(r)\qquad\mbox{ for all } x\in  \reald\mbox{  and } 0<r<r_0.
\end{equation} 
 Indeed, defining $B:=B(0,r)$ and $A:=B\cap B(x,r)$ we obtain,  by symmetry, that
$\int_A G_y(x)\,d\lam_B(y)=\int_A G_y(0)\,d\lam_B(y)$.
Further,  $G_y(x)\le g(r)\le G_y(0)$, if  $y\in B\setminus A$. Therefore 
$\int G_y(x)\,d \lam_B(y)\le \int G_y(0)\,d  \lam_B(y)$. 

And, last but not least, since  $g(r/2)\le g$ on $(0,r/2)$ and $d\int_0^{r/2}   s^{d-1}\,ds=(r/2)^d$,    
 we get the  following \emph{doubling property}:
There exists $1\le C_D\le 2^d C_G$ such that    
\begin{equation}\label{g-doubling}
       g(r/2) \le  C_D g(r),  \qquad \mbox{ for every } 0<r< r_0.
\end{equation}
So $\kap$ satisfies (\ref{doubling}) with $C=C_D$.

Another consequence of (\ref{g-decay}) is the following result (cf.\ \cite[Lemmas 2.5, 2.7]{mimica-vondracek}; 
if~$r_0=\infty$ and $|x|\ge r>0$, then  $g(|x|+r)/g(r)\ge C_D\inv g(|x|)/g(r)$).

\begin{proposition}\label{true-capacity}
For every $B:=B(0,r)$,  $0<r< r_0/4$, the following holds.
\begin{enumerate}
\item[\rm 1.]
For every $x\in \reald$,      
$$
\frac {g(|x|)}{g(r)}\ge  R_1^B(x)\ge C_G\inv \,  \frac{g(|x|+r)}{g(r)}.
$$
\item[\rm 2.] 
Let $\mu$ be the \emph{equilibrium measure} for~$B$, that is, $R_1^B=G\mu$.
Then 
$$
 C_G\inv \kap(r)\le \|\mu\|\le C_D \kap(r).
$$
\item[\rm 3.]
Property {\rm (\ref{rtwor})} holds with $a=C_D^{-2} C_G\inv$  and $\ve=r_0$.  
\end{enumerate}
\end{proposition} 

\begin{proof}   
The first inequality in (1) holds, since $G_0\in \es$ and $g(r)\inv G_0\ge 1$ on $B$.
Moreover, we know that $G\lam_B\in \es\cap \C(\reald)$ and $G\lam_B$ vanishes at infinity.
 By (\ref{vg-global}) and the minimum principle (or \cite[II.7.1]{BH}),  
\begin{equation}\label{R1} 
R_1^B\ge (C_Gg(r))\inv  G\lam_B. 
\end{equation} 
For every $x\in\reald$, we have  $g(|x-y|)\ge g(|x|+r)$, $y\in B$, and hence 
\begin{equation}\label{p-lower}
 G\lam_B(x)\ge  g(|x|+r). 
\end{equation} 
Clearly, (\ref{R1}) and (\ref{p-lower}) imply the second inequality in (1).

Using the doubling property (\ref{g-doubling}), we conclude from (\ref{p-lower})  that
\begin{equation}\label{p-lower-2}
       G\lam_B(x) \ge  g(|x|+r)\ge  g(2r) \ge C_D\inv   g(r), \qquad\mbox{for every }x\in B.
\end{equation} 
Moreover, $G\mu=R_1^B=1$ on $B$, hence, by Fubini's theorem and the symmetry of $G$,  
$$
         1=\int G\mu\,d\lam_B=\int\bigl(\int G(y, x)\,d\mu(x)\bigr) \,d\lam_B(y)=\int G\lam_B\,d\mu.
$$
Therefore (2) follows from (\ref{vg-global})  and (\ref{p-lower}).    

Finally,  
 by (1),  we have $R_1^B\ge C_G\inv g(4r)/g(r)\ge C_D^{-2} C_G\inv$ on $\partial B(0,3r)$.
\hbox{If~$x\in\reald$} and $x':=x+(0,\dots,0,3r)$, then $|x'-x|=3r$,  
$B(x',r)\subset  B(x,4r)\setminus \ov B(x,2r)$,  and hence, using translation invariance,  
$$
R_1^{B(x,4r)\setminus \ov B(x,2r)} (x)\ge R_1^{B(x',r)} (x) \ge C_D^{-2} C_G\inv.
$$
\end{proof} 
 
We now obtain the following. 

\begin{theorem}\label{iso} 
Let $\mathbbm P=(P_t)_{t>0} $ be a right continuous sub-Markov semigroup on~$\reald$, $d\ge 1$, 
such that the function $G_0: x\mapsto g(|x|) $ is $\mathbbm P$-excessive, where 
 $g>0$ is  decreasing, $0<g<\infty$ on $(0,\infty)$, $\lim_{r\to 0} g(r)=g(0)=\infty$, $\lim_{r\to \infty} g(r)=0$, 
and {\rm(\ref{g-decay})}~holds.
Further, suppose that the potential kernel $V_0$ \!of \,$\mathbbm P$ is given by convolution with $G_0$.  

\begin{enumerate}          
\item[\rm 1.] 
Let  $h\colon (0,1)\to (0,1)$,  $\lim_{t\to 0} h(t)=0$,  let $\delta>0$ and $\vp\in \C(\reald)$, $\vp >0$. 
Then  there exist a locally finite set $Z$ in $\reald$ and \hbox{$0<r_z<\vp(z)$}, 
\hbox{$z\in Z$},  
such that the closed balls $\ov B(z,r_z)$ are pairwise disjoint,
the union  of all $\ov B(z,r_z)$ is unavoidable, and 
$\sum_{z\in Z} \kap(r_z) h(r_z)<\delta$.
 \item[\rm 2.] 
Let $Z\subset \reald$ and $r_z>0$, such that the  union of all $\ov B(z,r_z)$,     
$z\in Z\subset \reald$, is~unavoidable. Then
\begin{equation}\label{nec}
\sum\nolimits_{z\in Z}  g(|z|) \kap ( r_z)=  \sum\nolimits_{z\in Z}  \frac{g(|z|)} {g(r_z)} =           \infty.
\end{equation} 
In particular,  $\sum_{z\in Z} \kap r_z=   \infty$. 
\end{enumerate}
\end{theorem}

\begin{proof}
1. Consequence of Theorem \ref{unavoidable-bal}.

2.  By Lemma \ref{simple},(c),  $\sum_{z\in Z} R_1^{\ov B(z,r_z)}(0)=\infty$. By Proposition~\ref{true-capacity},1
and translation invariance, $ R_1^{\ov B(z,r_z)}(0)\le g(|z|)/g(r_z)$, for every $z\in Z$. So (\ref{nec}) holds.

The proof will be finished using  Lemma \ref{simple},(b). However, not having assumed that the set $Z$ is locally finite,
we have to work a little.
Clearly, $\sum_{z\in Z} \kap (r_z)=   \infty$, unless the set of all $z\in Z$ such that  $\kap (r_z)\ge \kap(1)$ is finite. 
By Lemma \ref{simple},(b), it hence suffices to consider the case, where $\kap(r_z)<\kap(1)$ and hence $r_z<1$,
for every $z\in Z$. But then, of course, the balls $\ov B(z,r_z)$ with $|z|< 1$ are contained in~$B(0,2)$. Hence,
applying Lemma \ref{simple}(b) once more, we may assume without loss of generality that $|z|\ge 1$
for every $z\in Z$. Having (\ref{nec}) we now immediately see that  $\sum_{z\in Z} \kap r_z=   \infty$, since $g(|z|)\le g(1)$,
whenever $|z|\ge 1$.            
\end{proof}

\section{Application to censored stable processes}\label{section-censored}

Throughout this section let  $U$ be a (non-empty) bounded $C^{1,1}$ open set  in $\reald$, $d\ge 2$,
and $\a\in (1,2)$. Let  $\mathfrak X$ be the censored $\a$-stable process on $U$
(see \cite{bogdan-burdzy-chen, chen-kim, chen-kim-song}) and let $E_{\mathfrak X}$ denote the set of all excessive functions 
for $\mathfrak X$. 

We  claim that $(U, E_{\mathfrak X})$ is a balayage space satisfying the assumptions
made in Section \ref{champagne-bal}  
and that the following analogue of Theorems \ref{riesz-hausdorff} and  \ref{a-stable} holds.

\begin{theorem}\label{censored-hausdorff}
\begin{enumerate}  \item[\rm (1)] 
 Let $\psi\in \C(U)$,  $0<\psi\le \dist(\cdot,\uc)$,   
and $h\colon (0,1)\to (0,1)$ with
 $\lim_{t\to 0} h(t)=0$.
Then, for every $\delta>0$,
there is a locally finite set~$Z$ in~$U$ and \hbox{$0<r_z<\psi(z)$},
\hbox{$z\in Z$},  
such that the closed balls $\ov B(z,r_z)$ are pairwise disjoint,
the union  of all~$\ov B(z,r_z)$ is unavoidable in $U$, and 
$\sum_{z\in Z} r_z^{d-\a} h(r_z)<\delta$.
\item[\rm (2)] 
 Let $\phi$ be a measure function with \hbox{$\liminf_{t\to 0} \phi(t) t^{\a-d}=0$}.
Then there exists a~relatively closed set $A$ in $U$ such that $A$ is unavoidable, \hbox{$m_\phi(A)=0$}
and, for every connected component~$D$ of~$U$, the set~$D\setminus A$ is connected.
\item[\rm (3)]
There exists a relatively closed set $A$ in $U$ such that $A$ is unavoidable,
the Hausdorff dimension of $A$ is $d-\a$, and, for every connected component~$D$ of~$U$, 
the set~$D\setminus A$ is connected.
\end{enumerate} 
\end{theorem} 

For $x,y\in U$, let   
\begin{equation*} 
\gr(x,y):=|x-y|^{\a-d}   \und \du(x):=\dist(x,\uc).
\end{equation*} 
Let $V_0$ be the potential kernel of $\mathfrak X$. By \cite[p.\ 599 and Theorem 1.3]{chen-kim}, there exists 
a~unique (symmetric) function $G\colon U\times U\to (0,\infty]$ such that 
$G$ is continuous off the diagonal, $G=\infty$ on the diagonal, and
\begin{equation*} 
       \int_U G(\cdot,y) f(y)\,dy=V_0f,\qquad\quad f\in \B^+(U).
\end{equation*} 
Moreover, there exists $c>1$ such that, defining $\Psi(x,y):=\du (x) \du (y) |x-y|^{-2(\a-1)}$,
\begin{equation}\label{G-est}
 c\inv \min\{1, \Psi  \}\, G_0 \le G\le  c \min\{1,\Psi\}\, G_0   \on U\times U.
\end{equation} 

 In particular, if  $x\in U$, $\ve\in (0,1)$ with $\ve^{\a-1}<\du(x)/4$, and $y\in \ov B(x,2\ve)$, then  
\begin{equation}\label{G-riesz-est}
       c\inv  G_0(x,y) \le G(x,y)\le  c  G_0(x,y)
\end{equation} 
(we have  $\du(y)\ge \du(x)-\ve>\du(x)/2$ and hence $\Psi(x,y)>(1/2)\du(x)^2 \ve^{-2(\a-1)}>1$).

\begin{lemma}\label{GGmu} 
For every   measure $\mu$ on $U$ the following holds:
\begin{itemize}
\item[\rm (i)] $G\mu\le c \gr \mu$.
\item[\rm (ii)] If $z\in U$, $\ve\in (0,1)$ with $\ve^{\a-1}\le \du (z)/5$,     
and $\mu$ is supported by $\ov B(z,\ve)$,    then $\gr\mu\le c G\mu$ on $\ov B(z,\ve)$.
\item[\rm (iii)]  If $\gr\mu\in\C(U)$, then $G\mu\in \C(U)$.    
\end{itemize} 
\end{lemma}

\begin{proof} (i) and (ii) are immediate consequences of (\ref{G-est}) and (\ref{G-riesz-est}).
To prove (iii),  we introduce $\tilde G\colon U\times U\to \left(0,\infty\right]$ such that $\tilde G\ge \gr$       
and $G+\tilde G=(c+1)\gr$. Then the functions $G,\tilde G$ are lower semicontinuous.  
So,  for every measure $\mu$ on $U$, the functions $G\mu$, $\tilde G\mu$ are lower semicontinuous, by Fatou's  lemma,
and their sum is~$(c+1)\gr \mu$. Thus $G\mu\in \C(U)$, if $\gr \mu\in\C(U)$.      
\end{proof} 

\begin{proposition}\label{v-Feller}
The potential kernel $V_0$  is a strong Feller kernel, that is, $V_0(\B_b(U))\subset \C_b(U)$.    
\end{proposition} 

\begin{proof} Let $f\in\B_b^+(U)$. It is well known (and easily verified) that the function $\int_U \gr(\cdot,y)f(y) \,dy$ 
is continuous and bounded. So, by Lemma \ref{GGmu}, $V_0f\in \C_b(U)$.   
\end{proof} 

\begin{corollary}\label{censored-balayage space}
$(U,E_{\mathfrak X})$ is a balayage space.
\end{corollary} 

\begin{proof} 
Since $V_0(\B^+(U))\subset E_{\mathfrak X}$ (see \cite[II.3.8.2]{BH} or \cite[Proposition 2.2.11]{H-course}) 
and, for every $v\in E_{\mathfrak X}$, 
there exist  $f_n\in \B_b^+(U)$, $n\in\nat$,  such that $V_0f_n\uparrow v$ (see \hbox{\cite[II.3.11]{BH}}  or \cite[Theorem 2.2.12]{H-course}),
we see that $E_{\mathfrak X}$ satisfies property (ii) of (B$_4$). 

Next let $x,y\in U$, $x\ne y$. Since $G(x,y)<\infty=\liminf_{z\to y}  G(z,y)$, there exists $0<\ve<\du (y)$ such that 
$G(x,z)< G(y,z)$, for all $z\in B(y,\ve)$.  
Then $v:=V_01_{B(y,\ve)}\in E_{\mathfrak X}\cap \C(U)$   and $v(x)<v(y)$. 
Moreover, $v\to 0$ at infinity, by~(\ref{G-est}). 

Since $1\in E_{\mathfrak X}$, we conclude that $ E_{\mathfrak X}$ satisfies (B$_4$). Thus $(U,E_{\mathfrak X})$ is a balayage space,
by Corollary \ref{process-bal}.
\end{proof}

\begin{proof}[Proof of Theorem \ref{censored-hausdorff}]  
Let us first verify that the balayage space $(U,E_{\mathfrak X})$ satisfies the assumptions made in Section \ref{champagne-bal}.
To that end let $\rho_0$ be the diameter of $U$ and  $\kap(r):=r^{d-\a}$, $0<r\le \rho_0$.    
Clearly, $\kap$ is strictly increasing 
on $\left(0,\rho_0\right]$ and (\ref{doubling}) holds with $C:=2^{d-\a}$.   
Further, the estimate  (\ref{G-local}) follows immediately 
from (\ref{G-riesz-est}), and the representation (\ref{p-rep}) of potentials is given by Remark \ref{G-remarks}.3. 

The harmonicity
of $G(\cdot, y)$, $y\in U$, on $U\setminus \{y\}$ is already stated in \cite[p.\ 599]{chen-kim}. We may as well get it from
the fact that taking $r_n:=\delta_U(y)/n$,  $n\in\nat$, the functions 
$h_n:= (\lambda^d(B(y,r_n)))\inv V_01_{B(y,r_n)}$ are harmonic on $U\setminus \ov B(y,r_n)$ and converge to $G(\cdot,y)$
locally uniformly on $U\setminus \{y\}$. By (\ref{G-est}), each function $G(\cdot,y)$, $y\in U$, tends to zero at $\partial U$,
and hence  is a potential. 

To obtain (\ref{rtwor}), let us fix a compact $K$ in $U$ and choose $\ve\in (0,1)$ such that 
$\ve^{a-1}<\delta_U/5$ on $K$.     Let  $x\in K$, $0<r<\ve$,  and let $\mu$ denote the equilibrium 
measure for $A:=\ov B(x,(3/4)r)\setminus 
B(x,(2/3)r)$ with respect to Riesz potentials, that is,  $\mu$~is the (unique) measure on $A$ satisfying $G_0\mu\in \C(U)$    
and $G_0\mu=1$ on~$A$. 
By~translation and scaling invariance, the value $\b:=G_0\mu(x)$ does not depend on~$x$ and~$r$.
Of course, $G_0\mu\le 1$, by the minimum principle (for Riesz potentials). 
By Lemma~\ref{GGmu}, $p:=c\inv G\mu\in\mathcal P(U)$ and $p\le G_0\mu$.        
Hence, by the minimum principle (for~$(U,E_{\mathfrak X})$), 
$p\le R_1^{B(x,r)\setminus \ov B(x,r/2)}$.   In particular, $R_1^{B(x,r)\setminus \ov B(x,r/2)}(x)\ge c^{-2} \b$.

So we may apply Theorem \ref{KZ} and obtain (1) in Theorem \ref{censored-hausdorff}. To prove (2) let~$\phi$
be a measure function such that $\liminf_{t\to 0}\phi(t)t^{\a-d}=0$. By Theorem \ref{existence}, there exists a compact~$F$
in~$B(0,1)$ such that $B(0,1)\setminus F$ is connected, $m_{\phi}(F)$=0, and $F$~is nonpolar with respect to Riesz potentials.
So there exists a measure~$\nu\ne 0$ on~$F$ such that $G_0\nu\in\C(\reald)$ and $G_0\nu\le 1$. Let
\begin{equation*} 
               \g:=\inf\{ G_0\nu(x)\colon x\in \ov B(0,1)\}.
\end{equation*} 

We now choose a locally finite set $Z$ in $U$ and 
$0<r_z<\delta_U(z)/5$, $z\in Z$,  such that the closed balls $\ov B(z,r_z)$ are pairwise disjoint and the union of all $\ov B(z,r_z)$,
$z\in Z$, is unavoidable in $U$. Let $A$ be the union of all compact sets $F_z:=z+r_z F$, $z\in Z$. Clearly, $A$ is relatively closed in $U$
and, for every connected component $D$ of $U$, the set $D\setminus A$ is connected. Moreover, $m_\phi(A)=0$, since $m_\phi(F)=0$.

To prove that $A$ is unavoidable,  let $z\in Z$ and let  $T_z$ denote the transformation $x\mapsto z+r_z x$ on $\reald$.
Then the measure $\nu_z:=r_z^{d-\a} T_z(\nu)$ is supported by \hbox{$T_z(F)=F_z$},  $G_0\nu_z\in\C(\reald)$, $G_0\nu_z\le 1$ on $\reald$, 
and $G_0\nu_z\ge \g$ on $\ov B(z,r_z)$. By Lemma \ref{GGmu}, \hbox{$G\nu_z\in \C(U)$}, $c\inv G\nu_z\le  G_0\nu_z\le 1$ on $U$, and
$G\nu_z\ge c\inv G_0\nu_z\ge c\inv \g>0$ on $\ov B(z,r_z)$. Since $\nu_z$ is supported by $F_z$, we see, 
by the minimum principle, that $R_1^{F_z}\ge c\inv G\nu_z$.  In particular, 
\begin{equation*} 
 R_1^{F_z}\ge c^{-2}\g \on  \ov B(z,r_z).            
\end{equation*} 
Thus $A$ is unavoidable, by Proposition \ref{AB-unavoidable},2. 

As before, (3) is a consequence of (2) considering $\phi(t):=t^{d-\a}(\log^+\frac 1t)\inv$.
\end{proof} 

\begin{remark} {\rm
As in Theorem \ref{riesz-hausdorff}, the condition $\liminf_{t\to 0} \phi(t) t^{\a-d}=0$
in (2) of Theorem \ref{censored-hausdorff} is necessary for the statement.
}
\end{remark}

\section{Appendix}

\subsection{Balayage spaces}\label{sec-bal}

 Throughout this section let $X$ be a locally compact space with countable
base. As before, let  $\C(X)$ be the set of all continuous real functions on $X$, let
$\K(X)$ denote  the set of all functions in $\C(X)$ having compact support, and 
let $\B(X)$ be the set of all Borel measurable numerical functions on $X$.

In probabilistic terms, the theory of balayage spaces is the theory of Hunt processes
with proper potential kernel on $X$ such that every excessive function is the supremum of its continuous excessive minorants
and  there are two strictly positive continuous excessive functions $u,v$ such that $u/v$
vanishes at infinity    
 (Corollary~\ref{process-bal}).

We shall introduce balayage spaces by  properties of their positive hyperharmonic functions 
(see \cite[Definition 1.1.3]{H-course}) and give    characterizations in terms of excessive functions for sub-Markov resolvents   
(see \cite[II.3.11, II.4.7, II.7.8 ]{BH} and \cite[Theorem 2.2.12, Theorem 2.3.4, Corollary 2.3.7]{H-course})  and for sub-Markov semigroups
(see \cite[II.4.9 and II.8.6]{BH} and \cite[Corollary 2.3.8]{H-course}). 
For a characterization by properties of an associated family of harmonic kernels
(given by $H_U(x,\cdot):=\vx^\uc$) see \cite[III.2.8 and III.6.11]{BH} or
\cite[Theorems 5.1.2 and 5.3.11]{H-course}. Numerous examples are given in~\cite{BH}; see also \cite[Examples 10.1]{convexity}.  

Let $\W$ be a convex cone of positive lower semicontinuous numerical
functions on~$X$.  The coarsest topology on $X$ which is at least as
fine as the initial topology and for which all functions of $\W$ are continuous
is   the $(\W)$-\emph{fine topology}. For every function $v\colon X\to [0,\infty]$,
the largest finely lower semicontinuous  minorant of~$v$ is  denoted by~$\hat v^f$.

\begin{definition}\label{def-balspace}
 $(X,\W)$ is a \emph{balayage space}, if the 
following holds:
\begin{itemize} 
\item[{\rm (B$_1$)}] $\W$ is \emph{$\sigma$-stable}, that is,
   $\sup v_n\in\W$ for every increasing sequence $(v_n)$
  in $\W$.
\item[{\rm (B$_2$)}] $\widehat{\inf \V}^f\in\W$, for every  subset
  $\V$ of $\W$.
\item[{\rm (B$_3$)}] If $u,v',v''\in\W$ such that $u\le v'+v''$, there
exist $u',u''\in\W$ such that $u=u'+u''$, $u'\le v'$, and $u''\le v''$.
\item[{\rm (B$_4$)}] 
\begin{itemize}
\item[\rm(i)] 
$\W$ is linearly separating.
\footnote{That is,  for all $x,y\in X$, $x\ne y$, and $\lambda\in [0,\infty)$,
 there exists $v\in\W$ such that $v(x)\ne \lambda v(y)$.} 
\item[\rm(ii)] 
 For every $w\in \W$, $               w=\sup\{v\in\W\cap \C(X)\colon v\le w\} $. 
\item[\rm(iii)] 
There are strictly positive $u,v\in\W\cap \C(X)$ such that $u/v\to 0$  at~infinity.   
\end{itemize} 
\end{itemize} 
\end{definition}

A sub-Markov resolvent $\vvl$ on $X$ is a family of 
kernels~$V_\lambda$ 
on $X$ such that, for every $\lambda>0$, the kernel~$\lambda
V_\lambda$  is sub-Markov (that is, $\lambda V_\lambda 1\le 1$) and
$V_\lambda=V_\mu+(\mu-\lambda)V_\lambda V_\mu$, for all~\hbox{$\lambda,\mu\in (0,\infty)$}.  
 The~kernel  $V_0:=\sup_{\lambda>0} V_\lambda$ is called the \emph{potential kernel} of $\mathbbm V$. 
The resolvent $\mathbbm V$  is  \emph{right continuous}, if $\liml \lvl \vp=\vp $, for every $\vp\in\K(X)$.
It is \emph{strong Feller}, if $V_\lambda (\B_b(X))\subset  \C_b(X)$, for every $\lambda>0$. 

A  function $u\in \B^+(X)$  is  \emph{$\mathbbm V$-excessive},    
if $\sup_{\lambda>0} \lambda V_\lambda u= u$.  
Let  $\ev$ denote the set of all $\mathbbm V$-excessive functions.  
The convex cone $\ev$ contains $V_0(\B^+(X))$ and satisfies~(B$_1$).

\begin{theorem}\label{II.4.7}
For every sub-Markov resolvent  $\mathbbm V$ 
on $X$, the following statements are equivalent.
\begin{enumerate} 
\item[\rm 1.]  $(X,E_{\mathbbm V})$ is a balayage space.
\item [\rm 2.] The resolvent $\mathbbm V$ is right continuous, and  
        $E_{\mathbbm V}$  satisfies~$(B_4)$.
\end{enumerate} 
Moreover, if $(X,\ev)$ is a balayage space, then
$\liml \lvl(x,U)=1$  for all $x\in X$ and Borel measurable fine neighborhoods $U$  of $x$.
\end{theorem} 

\begin{theorem}\label{V-approx} Suppose that $\mathbbm V$ is a sub-Markov resolvent    
such that its potential kernel $V_0$ is \emph{proper}, that is, there exists $g\in \B^+(X)$, $g>0$,  such that $V_0g<\infty$. 
Then $\ev$ is the set of all limits of increasing sequences in $V(\B^+(X))$,  and $\ev$ is linearly separating.  
\end{theorem}

\begin{corollary}\label{II.4.7-Feller}
Let $\vvl$ be a right continuous strong Feller sub-Markov resolvent on $X$ such that $V_0$ is proper
and there are strictly positive functions \hbox{$u,v \in \ev\cap \C(X)$} such that  $u/v\to 0$  at infinity.
Then $(X,\ev)$ is a~balayage space. 
\end{corollary}

A family $\ppt$ of kernels on $X$ is a \emph{sub-Markov semigroup} if
$P_t 1\le 1$ and $P_{s+t}=P_s P_t$, for all $s,t>0$.
It is \emph{right continuous} if $\lim_{t\to 0} P_t  \vp=\vp $, for all \hbox{$\vp\in\K(X)$}.
A  function $u\in\B^+(X)$    is  \emph{$\mathbbm P$-excessive}, 
\hbox{if~$\sup_{t>0} P_t u= u$}. 
Let  $\es$ denote the set of all $\mathbbm P$-excessive functions. 
If $\mathbbm P$ is right continuous (measurability of $(t,x)\mapsto P_tf(x)$, $f\in\K(X)$,  would be sufficient), the connection
to an associated sub-Markov resolvent $\vvl$ is given by $V_\lambda=\int_0^\infty
e^{-\lambda t} P_t \,dt$, $\lambda>0$, we have $E_{\mathbbm P}=E_{\mathbbm V}$   
(see \cite[II.3.13]{BH} or \cite[Corollary 2.2.14]{H-course}), 
and the kernel $V_0=\sup_{\lambda>0} V_\lambda=\int_0^\infty P_t\,dt$ is also called the \emph{potential kernel of~$\mathbbm P$}. 

\begin{corollary}\label{II.4.9}
For every sub-Markov semigroup $\ppt$ on~$X$ the following
holds.
\begin{enumerate} 
\item [\rm 1.] $(X,\es)$ is a balayage space if and only if $\mathbbm P$ is
  right continuous and $\es$  satisfies~$(B_4)$.
\item [\rm 2.] 
If \,$\mathbbm P$ is right continuous, the potential kernel of \,$\mathbbm
P$ is proper, the resolvent~$
\mathbbm V$\! of~\,$\mathbbm P$ {\rm(}or even $\mathbbm P$ \!itself\,{\rm)}
 is strong Feller, and there are strictly positive functions   \hbox{$u,v\in\es\cap \C(X)$} such that
 $u/v\to 0$ at infinity, then $(X,\es)$ is a balayage space. 
\end{enumerate} 
\end{corollary} 

Finally, given a Hunt process   $\mathfrak X=(\Omega,\mathfrak M, \mathfrak M_t, X_t,\theta_t,P^x)$ on $X$,
the transition kernels $P_t$, $t>0$, defined by $P_tf(x):=E^x(f\circ X_t)$, $f\in \B^+(X)$, form a right
continuous sub-Markov semigroup $\mathbbm P$ on $X$.  
By definition, $E_{\mathfrak X}:=E_{\mathbbm P}$, and $V_0=\int_0^\infty P_t\,dt$ is the \emph{potential kernel of~$\mathfrak X$}.

\begin{corollary}\label{process-bal}
Let $\mathfrak X=(\Omega,\mathfrak M, \mathfrak M_t, X_t,\theta_t,P^x)$ be a Hunt process on $X$.
Then $(X,E_{\mathfrak X})$ is a balayage space if and only if $E_{\mathfrak X}$ satisfies $(B_4)$.
\end{corollary}

Conversely, the following holds (see \cite{BH}).     

\begin{theorem}\label{bal-res}
Let $(X,\W)$ be a balayage space such that $1\in\W$,    and let \hbox{$p\in \Px$} be a bounded strict potential.\footnote{See 
Section 2 for definitions.} 
Then there exists a unique right continuous strong Feller sub-Markov resolvent $\vvl$ on $X$ such that $V_01=p$  and $\ev=\W$.
Moreover, $\mathbbm V$ is the resolvent of a right continuous sub-Markov semigroup $\ppt$ on $X$, and $\mathbbm P$
is the transition semigroup of a Hunt process $\mathfrak X$ on $X$.
\end{theorem}

\begin{remark}
{
The assumption $1\in\W$ is not very restrictive. Indeed, let  $(X,\W)$ be an arbitrary balayage space,
let $v\in \W\cap \C(X)$, $v>0$, and $\widetilde \W:= \{u/v\colon u\in \W\}$.  
 Then $(X,\widetilde \W)$  is a balayage space such that $1\in \widetilde \W$.
}
\end{remark}

Finally, let us mention the possibility of constructing new examples by subordination         
with  convolution semigroups~$(\mu_t)_{t>0}$ on $(0,\infty)$, that is, families
of measures~$\mu_t$ on~$(0,\infty)$ such that $\|\mu_t\|\le 1$, $\mu_s\ast\mu_t=\mu_{s+t}$,  
for all $s,t\in (0,\infty)$, and $\lim_{t\to 0} \mu_t=\ve_0$ (that is, $\lim_{t\to 0} \mu_t(\vp)=\vp(0)$, for all $\vp\in \K((0,\infty))$).
The following result is contained \cite[V.3.6, V.3.7]{BH}). 

\begin{theorem}\label{subordination}
Let $(\mu_t)_{t>0}$ be a convolution semigroup on $(0,\infty)$ such that 
\hbox{$\kappa:=\int_0^\infty \mu_t\,dt$}  is a  Radon measure 
which is absolutely continuous with respect to Lebesgue measure on $(0,\infty)$. 

Moreover, let $\mathbbm P$ be a sub-Markov semigroup on $X$ with strong Feller resolvent
 such that $(X,\es)$ is a balayage space, and define  kernels $P_t^\mu$, $t>0$, by 
$$
       P_t^\mu f:=\int P_sf\,d\mu_t(s), \qquad  f\in\B^+(X).
$$
Then  $\mathbbm P^\mu=(P_t^\mu)_{t>0}$ is a sub-Markov semigroup on $X$ with strong Feller resolvent,   
and $(X,E_{\mathbbm P^\mu})$ is a balayage space.\footnote{If $\|\mu_t\|=1$, for every $t>0$, 
and $\mathbbm P$ is a Markov semigroup, then, of course, $\mathbbm P^\mu$ is a Markov 
semigroup as well.}                                   
\end{theorem}

\subsection{Nonpolar compact sets of Cantor type}
 For the convenience of the reader, we first present a self-contained construction of small nonpolar compact sets 
of Cantor type for classical potential theory as well as  for Riesz potentials on Euclidean space 
(the result itself is a special case of \hbox{\cite[Theorem 5.4.1]{adams-hedberg})}.

Let $d\ge 1$ and $0<\a\le 2$ with $\a<1$ if $d=1$, and let    
$$
                     \kap (r):=(\log\frac 1r)\inv \und G(x,y):=\log \frac 1{|x-y|},
$$
if  $\a=d=2$. In the other cases, let 
\begin{equation}\label{normal} 
                     \kap (r):=r^{d-\a} \und G(x,y):=|x-y|^{\a-d}. 
\end{equation} 

\begin{theorem}\label{existence}
Let $d\ge 2$ and let $\phi$ be a measure function such that 
$$\
\liminf\nolimits_{t\to 0} \phi(t)/\kap(t)=0.
$$
Then there is a~nonpolar compact $F\subset B(0,1)$ of Cantor type such that $m_\phi(F)=0$ and 
$B(0,1)\setminus F$ is  connected.
\end{theorem}

Let us first establish a simple scaling property for potentials of Lebesgue measure on cubes 
(see (\ref{capQ}) below).
Let 
$$
 K:=[- (2\sqrt d)\inv, (2 \sqrt d)\inv]^d
$$ 
so that the diagonal of $K$ has length $1$. For every $a>0$, let 
$$
     \Q_a:=\{ x+ a K\colon x\in \reald\}, 
$$
and, for  every  cube $Q$ in $\reald$, let $\mu_Q$ denote normalized Lebesgue measure on $Q$.
There exists a constant $c=c(d)$  such that    
\begin{equation}\label{capQ}
           c\inv  \le  \kap(a)\, \|G\mu_Q\|_\infty\le c   \qquad
                    (0<a\le 1/e,\,Q\in \Q_a) .
\end{equation} 
Indeed, let us assume for the moment that $\a=d=2$.   
 Since the function $t\mapsto t\log\frac 1t$
is increasing on~$\left(0,1/e\right]$, we obtain that, for all $0<\alpha <\b\le 1/e$,
\begin{equation}\label{polco}
 2\pi    \b^2 \log\frac 1\b \ge  2\pi\int_{\alpha }^\b t\log\frac 1t \,dt  =  \int_{B(0,\b)\setminus B(0,\alpha )} \log\frac 1{|x|} \, dx
                     \ge 2\pi  (\b-\alpha )  \alpha \log\frac 1{\alpha }.  
\end{equation} 
Knowing that $B(0,a/(2\sqrt 2))\setminus B(0,a/(4\sqrt 2))\subset aK$ and $aK\subset B(x,a)$, for $x\in aK$, the estimate
(\ref{capQ}) follows easily from (\ref{polco}).
In the other cases, for every $a>0$,
 $\|G\mu_{x+aK}\|_\infty=\|G\mu_{aK}\|_\infty=a^{\a-d}\|G\mu_K\|_\infty$, 
since  $\mu_{aK}$ is the image of $\mu_K$ under the scaling $x\mapsto ax$. 

We prove Theorem \ref{existence} by a recursive construction of a decreasing sequence
of finite unions $F_m$ of  cubes and probability measures $\mu_m$ on $F_m$, $m\in \nat$. 
Of course, we shall finally define  $F:=\bigcap_{m\in \nat} F_m$.

Let $K_1:=(1/e)K$ and $c_1:=\|G\mu_{K_1}\|_\infty$.             
We start with $F_1=K_1$ and the measure $\mu_1:=\mu_{K_1}$. Let us suppose
that $m\in \nat$ and that after $m-1$ construction steps we have a probability measure
$$
                      \mu_m= \frac 1M\sum\nolimits_{1\le i\le M}  \mu_{Q_i}    \on \   F_m=Q_1\cup \dots \cup Q_M,
$$
where $M\in\nat$, $0<a\le 1/e$ and $Q_1,\dots, Q_M\in \Q_a$ are pairwise disjoint such that, for every $1\le i\le M$,
\begin{equation}\label{est-mq}
                                                       \frac 1M \|\mu_{Q_i}\|_\infty \le 2^{-(m-1)} c_1 
\end{equation} 
(true for $m=1$ with $M=1$ and  $a=1/e$). 

Our $m$-th construction step is the following: 
For $n\in\nat$ and $0<r<(1/2)a/n$ (to be fixed below),  we ``cut'' each $Q_i$ into $n^d$ cubes 
$Q_{i1},\dots,Q_{in^d}$ in~$\Q_{a/n}$  in the canonical way. For each $1\le j\le n^d$, let  $Q_{ij}'$ be the cube in $\Q_r$
having the same center as $Q_{ij}$ (see Figure 1).
\begin{center}
\includegraphics[width=13cm]{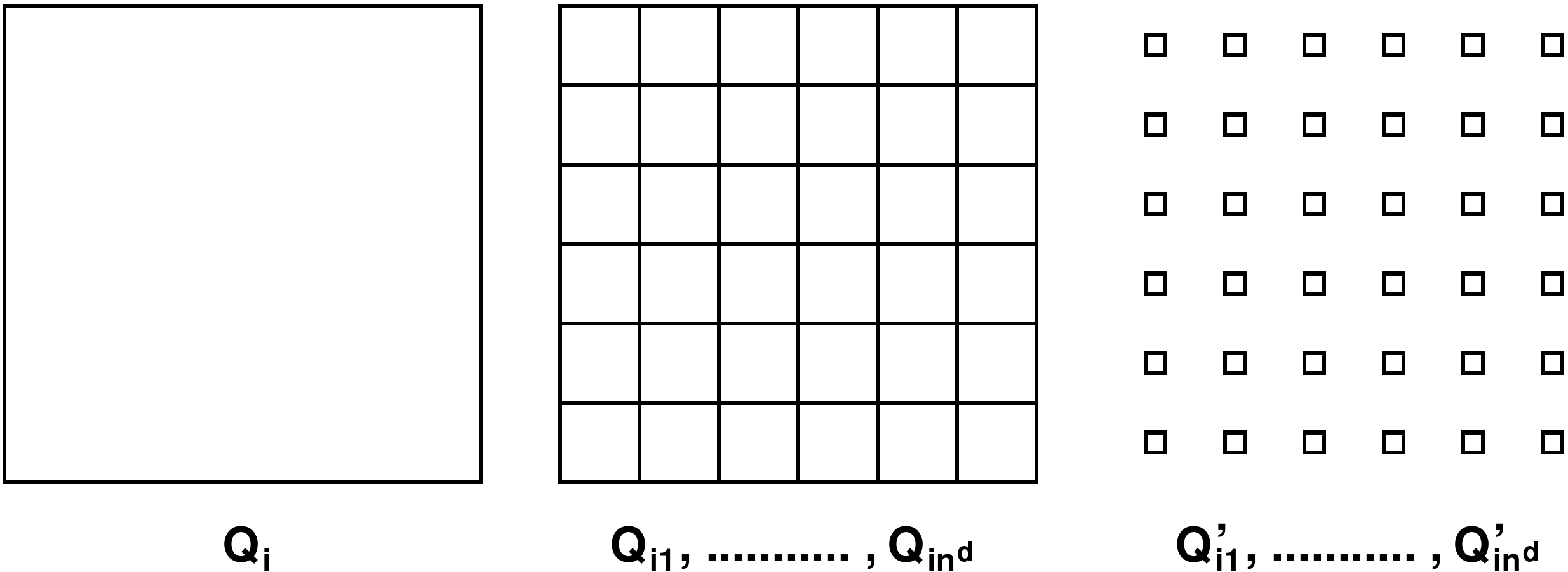}\\[2mm]
\raisebox{0cm}{Figure 1. The construction step}
\end{center} 
 Finally, let
$$
\mu_{m+1}:=\frac 1{M} \sum\nolimits_{1\le i\le M} \nu_i, \qquad\mbox{ where } 
 \nu_i:= \frac 1{n^d} \sum\nolimits_{1\le k\le n^d} \mu_{ Q_{ik}'}.
$$
We note that each $\nu_i$ is a probability measure on $Q_i$, $1\le i\le M$,  and that $\mu_{m+1}$ is a~probability measure on
 \begin{equation}\label{fm1}
F_{m+1}:=\bigcup\nolimits_{1\le i\le M,\, 1\le j\le n^d} Q_{ij}'.
\end{equation} 
Let $h:=\phi/\kap$.   
We may  choose $n\in\nat$ and $r\in (0,1/m)$    such that the following holds:
\begin{itemize}
\item[\rm (i)]
For all $i,j\in\{1,\dots, M\}$, $i\ne j$, $|G\nu_j-G\mu_{Q_j}|< 2^{-m} c_1$ on $Q_i$,
\item [\rm (ii)]
 $h(r)< 1/(3mM c^2\kap(a)) $,
\item [\rm (iii)]
$2c^2\kap(a)<n^d\kap(r)<3c^2\kap(a)$ and $2r<a/n$.
\end{itemize} 
Indeed, by uniform approximation,  there exists $n_0\in\nat$ such that (i) holds if $n\ge n_0$. 
We may assume without loss of generality that    $n^d/(n-1)^d< 3/2 $ and $3c^2\kap(a)
<n^{d}\kap(a/(2n))$ for all $n\ge n_0$.  Since $\lim_{t\to 0} \kap(t)=0$ and $\liminf_{t\to 0} h(t)=0$, 
there exists $r\in (0,1/m)$ such that $n_0^d\kap(r)< c^2\kap(a)$ and  (ii)~holds.  Let $n\in\nat$ be minimal
such that $n^d\kap(r)>2 c^2\kap(a)$. Then $n>n_0$, $(n-1)^d\kap(r)<2c^2\kap(a) $,  and
hence $n^d\kap(r)< (3/2) (n-1)^d\kap(r)<3 c^2\kap(a) $. Moreover, 
$\kap(r)<3c^2\kap(a)n^{-d}<\kap(a/(2n))$, and hence $r<a/(2n)$.
So (iii) holds as well.
      
Since each $Q_{ij}'$ is contained in an open ball of radius $r<1/m$, we obtain, by~(ii) and~(iii), that  
\begin{equation}\label{mphi}
      M_\phi^{1/m}(F_{m+1})\le M n^d \phi(r)= M n^d \kap(r) h(r)  < 3 Mc^2\kap(a) h(r)<1/m. 
\end{equation} 
                 
For the moment, let us fix $1\le i\le M$, $1\le j\le n^d$, and consider 
\begin{equation}\label{qij}
Q:=Q_{ij}'\in \Q_r.
\end{equation} 
Let $1\le k\le n^d$,    $k\ne j$. If $x\in Q$, $y\in Q_{ik}$ and $y'\in Q_{ik}'$ (see Figure 2 for a case, where $Q$ and
$Q_{ik}$ are close to each other), then $|x-y'|\ge (2\sqrt d)\inv a/n$         
and $|y'-y|\le a/n$, hence $|x-y|\le |x-y'|+|y'-y|\le (1+2\sqrt d) |x-y'|$ and 
$
|x-y'|^{-1}\le (1+2\sqrt d) |x-y|\inv
 $.
\begin{center}
\includegraphics[width=8cm]{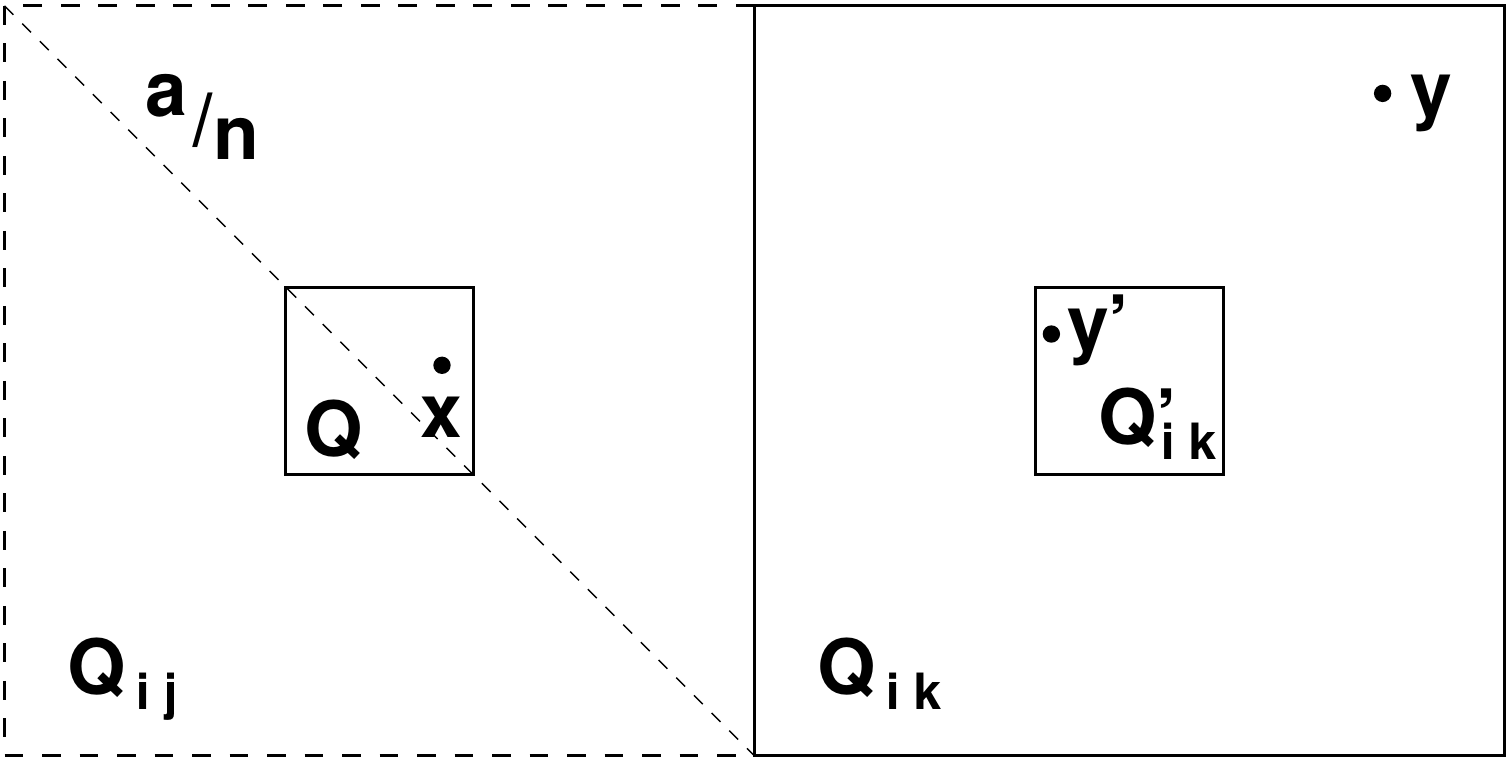}\\[2mm]
\raisebox{0cm}{Figure 2. The cubes $Q$, $Q_{ik}$ and $Q_{ik}'$} 
\end{center} 
If $d=2$, then $(1+2\sqrt d) |x-y|\inv\le e^2 |x-y|\inv \le |x-y|^{-3}$.  
Hence, defining   $C:= \max\{3,(1+2\sqrt d)^{d-\a} \}$, we obtain that
$G\mu_{Q_{ik}'}\le  C G\mu_{Q_{ik}}$ on $Q$. Thus    
\begin{equation}\label{GGG}
                     G\nu_i =  n^{-d} G\mu_Q   + n^{-d}\sum\nolimits_{1\le k\le n^d, k\ne j}  G\mu_{Q_{ik}'}
                   \le    n^{-d} G\mu_Q   + C G\mu_{Q_i}      \on Q,.
\end{equation} 
By our induction hypothesis (\ref{est-mq}),                                    
we  have  $(1/M) \|\mu_{Q_i}\|_\infty \le 2^{-(m-1)} c_1 $,  and hence,  by (\ref{capQ}) and  (iii),
\begin{equation*} 
     \bigl \| \frac 1{Mn^d} \,G\mu_Q\bigr\|_\infty \le \frac{c} {Mn^d\kap(r)}     
     < \frac  {c\inv} {2 M\kap(a)}\le \frac 12\, \bigl\|\frac 1M\,G\mu_{Q_i}\bigr\|_\infty\le  2^{-m}c_1
\end{equation*} 
(which allows us to continue our construction). So  (\ref{GGG}) leads to the inequality
\begin{equation*}
\frac 1M G\nu_i\le 2^{-m}c_1+  C2^{-(m-1)}c_1  \on Q.
 \end{equation*} 
By (i),  we know that
\begin{equation*} 
  G\mu_{m+1}=\frac 1M\left(G\nu_i+\sum\nolimits_{j\ne i} G\nu_j\right) \le \frac 1M G\nu_i+ G\mu_m+ 2^{-m} c_1  \on Q_i.
\end{equation*} 
Therefore  $G\mu_{m+1}  \le G\mu_m + (C+1)2^{-(m-1)} c_1$ on $Q$. Recalling the definitions   of $Q$ and $F_{m+1}$
(see (\ref{qij}) and (\ref{fm1})) and using the minimum principle, we finally see  that
\begin{equation}\label{gmgm}
    G\mu_{m+1}  \le G\mu_m + (C+1)2^{-(m-1)} c_1.
\end{equation} 
Clearly, (\ref{gmgm}) implies that the sequence $(G\mu_m)$ is bounded. As announced above,  let
$F$~denote the intersection of the decreasing sequence~$(F_m)$. Since the sequence $(\mu_m)$ is weakly
 convergent to a~probability measure~$\mu$      
on~$F$ satisfying $G\mu\le \sup_{m\in\nat} G\mu_m$, we obtain that 
$F$ is nonpolar. By  (\ref{mphi}), $m_\phi(F)=0$ finishing the proof of Theorem~\ref{existence}.

\bibliographystyle{plain} 
\def\cprime{$'$} \def\cprime{$'$}

{\small \noindent 
Wolfhard Hansen,
Fakult\"at f\"ur Mathematik,
Universit\"at Bielefeld,
33501 Bielefeld, Germany, e-mail:
 hansen$@$math.uni-bielefeld.de}\\
{\small \noindent Ivan Netuka,
Charles University,
Faculty of Mathematics and Physics,
Mathematical Institute,
 Sokolovsk\'a 83,
 186 75 Praha 8, Czech Republic, email:
netuka@karlin.mff.cuni.cz}

\end{document}